\documentclass[dvipsnames, reqno,11pt]{amsart}

\usepackage[a4paper,left=25mm,right=25mm,top=30mm,bottom=30mm,marginpar=20mm]{geometry} 
\usepackage{amsmath}
\usepackage{amssymb}
\usepackage{amsthm}
\usepackage{amscd}
\usepackage[final]{graphicx}
\usepackage{tikz}
\usetikzlibrary{calc}
\usepackage{bbm}
\usepackage[colorlinks=true, linkcolor=blue, citecolor=blue]{hyperref}

\renewcommand{\epsilon}{\varepsilon}

\numberwithin{equation}{section}

\newtheoremstyle{thmlemcorr}{10pt}{10pt}{\itshape}{}{\bfseries}{.}{10pt}{{\thmname{#1}\thmnumber{ #2}\thmnote{ (#3)}}}
\newtheoremstyle{thmlemcorr*}{10pt}{10pt}{\itshape}{}{\bfseries}{.}\newline{{\thmname{#1}\thmnumber{ #2}\thmnote{ (#3)}}}
\newtheoremstyle{defi}{10pt}{10pt}{\itshape}{}{\bfseries}{.}{10pt}{{\thmname{#1}\thmnumber{ #2}\thmnote{ (#3)}}}
\newtheoremstyle{remexample}{10pt}{10pt}{}{}{\bfseries}{.}{10pt}{{\thmname{#1}\thmnumber{ #2}\thmnote{ (#3)}}}
\newtheoremstyle{ass}{10pt}{10pt}{}{}{\bfseries}{.}{10pt}{{\thmname{#1}\thmnumber{ A#2}\thmnote{ (#3)}}}

\theoremstyle{thmlemcorr}
\newtheorem{theorem}{Theorem}
\numberwithin{theorem}{section}
\newtheorem{lemma}[theorem]{Lemma}
\newtheorem{corollary}[theorem]{Corollary}
\newtheorem{proposition}[theorem]{Proposition}

\theoremstyle{thmlemcorr*}
\newtheorem{theorem*}{Theorem}
\newtheorem{lemma*}[theorem]{Lemma}
\newtheorem{corollary*}[theorem]{Corollary}
\newtheorem{proposition*}[theorem]{Proposition}
\newtheorem{problem*}[theorem]{Problem}
\newtheorem{conjecture*}[theorem]{Conjecture}

\theoremstyle{defi}

\theoremstyle{remexample}
\newtheorem{remark}[theorem]{Remark}

\theoremstyle{ass}

\newcommand{\Acal}{\mathcal{A}}

\newcommand{\Ecal}{\mathcal{E}}

\newcommand{\Ical}{\mathcal{I}}

\newcommand{\Mcal}{\mathcal{M}}

\newcommand{\Ocal}{\mathcal{O}}

\newcommand{\Scal}{\mathcal{S}}

\DeclareMathOperator{\curl}{curl}
\DeclareMathOperator{\dist}{dist}

\newcommand{\norm}[1]{\|#1\|}

\newcommand{\abs}[1]{|#1|}

\newcommand{\dd}{\;\mathrm{d}}

\newcommand{\N}{\mathbb{N}}
\newcommand{\R}{\mathbb{R}}

\newcommand{\Z}{\mathbb{Z}}
\newcommand{\T}{\mathbb{T}}

\newcommand{\loc}{\mathrm{loc}}

\newcommand{\weakly}{\rightharpoonup}

\newcommand{\eps}{\epsilon}

\newcommand{\ffi}{\varphi}
\newcommand{\Wmin}{\overline{W}}

\DeclareMathOperator{\Curl}{Curl}

\DeclareMathOperator{\SO}{SO}

\def\XXint#1#2#3{{\setbox0=\hbox{$#1{#2#3}{\int}$} 
\vcenter{\hbox{$#2#3$}}\kern-.5\wd0}}

\DeclareMathOperator{\Id}{Id}
\DeclareMathOperator{\Diss}{Diss}
\AtBeginDocument{}
\newcommand{\ui}[1]{^{\left(#1\right)}}

\newcommand\restrict[1]{\raisebox{-.5ex}{$|$}_{#1}}
\newcommand{\qand}{\quad\text{and}\quad}
\newcommand{\vecr}[1]{\left(\hspace*{-.9ex}\begin{array}{r}#1\end{array}\hspace*{-.9ex}\right)}
\newcommand{\RS}{M}

\title[Dimension reduction in plasticity]{Asymptotic analysis of single-slip crystal plasticity in the limit of vanishing thickness and rigid elasticity}

\author{Dominik Engl}
\address{Mathematisch-Geographische Fakult\"at, Katholische Universit\"at Eichst\"att-Ingolstadt, Osten\-stra{\ss}e 28, Eichst\"att, 85072, Germany}
\email{dominik.engl@ku.de}

\author{Stefan Kr\"omer}
\address{Institute of Information Theory and Automation, Czech Academy of Sciences, Pod vod\'arenskou ve\v{z}\'i 4, CZ-182 08, Prague 8, Czech Republic}
\email{skroemer@utia.cas.cz}

\author{Martin Kru\v{z}\'ik}
\address{Institute of Information Theory and Automation, Czech Academy of Sciences, Pod vod\'arenskou ve\v{z}\'i 4, CZ-182 08, Prague 8, Czech Republic}
\email{kruzik@utia.cas.cz}

\begin{document}

%%%%%%%%%%%%%%%%%%%%%% ABSTRACT %%%%%%%%%%%%%%%%%%%%%%%%%%%%%%%%%
\begin{abstract}\vspace{-12pt} 
We perform via $\Gamma$-convergence a 2d-1d dimension reduction analysis of a single-slip elastoplastic body in large deformations. Rigid plastic and elastoplastic regimes are considered. In particular, we show that limit deformations can essentially freely bend even if subjected to the most restrictive constraints corresponding to the elastically rigid single-slip regime.
The primary challenge arises in the upper bound where the differential constraints render any bending without incurring an additional energy cost particularly difficult.
We overcome this obstacle with suitable non-smooth constructions and prove that a Lavrentiev phenomenon occurs if we artificially restrict our model to smooth deformations.
This issue is absent if the differential constraints are appropriately softened.

\vspace{8pt}


\noindent\textsc{Date:} \today.
\end{abstract}

\maketitle

%%%%%%%%%%%%%%%%%%%%%%%%% INTRODUCTION %%%%%%%%%%%%%%%%%%%%%%
%\tableofcontents
\section{Introduction}
The elastoplastic behavior of a crystalline solid under the action of external loads results from a combination of reversible elastic and irreversible plastic effects \cite{GuFrAn10MTC, Lee69}.
The state of the body is specified in terms of its deformation $v:\Omega\to\R^n$ from a reference configuration $\Omega\subset\R^n$. 
Elastic and plastic effect are classically assumed to combine via the Kr\"{o}ner-Lee-Liu multiplicative decomposition of the total strain $\nabla v = F_{\rm el}F_{\rm pl}$, see \cite{Kro60, LeeLiu67FSEP, Lee69}. 
Here, the elastic strain  $F_{\rm el}$ describes elastic response of the material while $F_{\rm pl}$ records the accumulation of plastic distortion.  
This decomposition and alternative modeling assumptions have recently been discussed in, e.g., \cite{CDOR18, DaF15, Del18}. 
It is usually considered that plastic distortion causes no volume change, i.e.,  ${\rm det} F_{\rm pl}=1$; cf.~\cite{SimHug98CI}. 
Crystal plasticity assumes that any deformation that is applied to a material is accommodated by the process of slip, where dislocation motion occurs on a slip plane. 
In this article, we deal with a single slip which means that $F_{\rm pl}$ differs from the identity $\Id\in\R^{n\times n}$ by a rank-one and  traceless matrix, i.e.,  $F_{\rm pl}=\Id+\gamma s\otimes m$ where $s,m\in\R^n$ are unit and mutually perpendicular vectors denoting the slip direction and the slip-plane  normal, respectively, and $\gamma$ measures the amount of slip. 
Note that ${\rm det} F_{\rm pl}=1$ always holds. 
Elastoplastic evolution results from the competition of elastic-energy storage and plastic-dissipation mechanisms. 
A common and powerful approach to the description of elasto-plastic evolutionary  behavior is via variational methods and semidiscretization in time, see e.g.~\cite{CHM02}.  
Thus, we can define a condensed-energy-density  function arising from a time-incremental problem as 
\begin{align}\label{eq:energydensity}
	W(F) = \inf_{F=F_{\rm el}F_{\rm pl}}\big(W_{\rm el}(F_{\rm el}) + W_{\rm pl}(F_{\rm pl}) + \Diss(F_{\rm pl})\big).
\end{align}
Here, $W_{\rm el}:\R^{n\times n}\to[0,\infty]$ denotes the elastic stored energy density (i.e. a potential of the first Piola-Kirchhoff stress), $W_{\rm pl}:\textrm{SL}(n)\to\R$ is the defect energy (see e.g.~\cite{GuFrAn10MTC}), and $\Diss:\textrm{SL}(n)\to\R$ represents the density of energy dissipated if we change the plastic strain from the identity (i.e., purely elastic deformation) to  $F_{\rm pl}$ in one time step  of the rate-independent plastic evolution.
We refer to \cite{MieRou15RIST} for more details.

The elastoplasticity is occasionally modeled as  elastically rigid, meaning  that  $W_{\rm el}(F_{\rm el})$ is finite only if $F_{\rm el}\in \SO(n)$, i.e., if the elastic deformation is a rotation, see e.g.~\cite{CoT05}. 
We refer in this work to such situations as to {\it the hard constraint}. Usually, the stored energy density is assumed to be proportional to the distance of the right Cauchy-Green strain $C_{\rm el}=F_{\rm el}^\top F_{\rm el}$ from the identity, or in other words, to the distance of $F_{\rm el}$ to the set of proper rotations. This will be for us {\it  the soft constraint}. 
For additional information and some generalizations the reader is referred to \cite{Con06, CDK11, CoT05}, see also \cite{CoD20,CoD21} for more recent results.

In this article, we will study the 2d-1d dimension reduction problem associated with the static minimization problem with a plastic energy density subject to either the soft or the hard constraint, with the main result given by Theorem~\ref{theo:hard_constraint} phrased in the language of Gamma-convergence. 
As we will see, a central difficulty in its proof, more precisely, for the construction of recovery sequences, is precisely the potential rigidity of the hard constraint. 
In particular, this makes it hard to locally bend in any way without incurring additional energy cost. 
By contrast, if the constraint is softened enough, this is no longer an issue as shown in Proposition~\ref{prop:ratio}. 
We overcome the challenges stemming from the strict differential constraints with a construction specifically tailored to our single-slip shear constraint, in essence given in Lemma~\ref{lem:angle_e1e2} (see also Figure~\ref{fig:bend1} and \ref{fig:bend2}).
We also show that this or a similar nonsmooth mode of bending is crucial for the hard constrained model in the sense that a Lavrentiev phenomenon occurs in the dimension reduction problem if we artificially restrict our model to smooth deformations, at least in the special case $s=\pm e_1$ (slip along the membrane): On the one hand, for smooth deformations, bending is impossible without paying energy of the order of the membrane thickness $h$ as shown in Proposition~\ref{prop:smoothturn}, where bending is forced by ``short'' boundary conditions.
On the other hand, general (nonsmooth) admissible deformations can approach short (and therefore bent) limit deformations with a cost of $o(h)$, which is reflected in the fact that limit membrane energy of all short deformations after dimension reduction is zero by Theorem~\ref{theo:hard_constraint} (in case $s=\pm e_1$; see also Lemma~\ref{lem:Wmin}(a)). 
Here, note that our hard constraint is much more restrictive than, say, the natural constraints on nonlinear elasticity which are also known to cause Lavrentiev phenomena, but so far only in very specific circumstances \cite{FoHruMi03a}.

\subsection{Setup of the problem and main results}

For $h>0$ we set $\Omega_h := (0,L)\times (-h,h)$ as the reference configuration of a two-dimensional thin beam with length $L>0$ and thickness $2h>0$.
The slip direction and slip-plane normal shall be given by $s,m\in \Scal^1$ with $m=s^\perp= R_{\frac{\pi}{2}}s$. 
Then, we consider the energies
\begin{align}\label{thin_energy}
	\Ecal_{h}: W^{1,2}(\Omega_h;\R^2) \to [0,\infty],\: v\mapsto \int_{\Omega_h} W(\nabla v) \dd y,
\end{align}
with
\begin{align}\label{constrained_density}
	W(F) = \begin{cases}
				|Fm|^2-1 &\text{ if } F\in\Mcal_{s},\\
				\infty &\text{ otherwise},
			\end{cases}\quad F\in\R^{2\times 2},
\end{align}
where the set $\Mcal_{s}$ is consists of all rotated shears in direction $s$ with slip-plane normal $m$, i.e.,
\begin{align*}
	\Mcal_{s}:= \{F=R(\Id + \gamma s\otimes m) : R\in\SO(2),\, \gamma \in \R\} = \{F\in \R^{2\times 2} : \det F = 1, |Fs|=1\}.
\end{align*}
It is easy to see that 
\begin{align}\label{density_slip}
	|Fm|^2-1=\gamma^2\quad \text{if $F=R(\Id + \gamma s\otimes m)\in\Mcal_{s}$ for $R\in\SO(2)$ and $\gamma\in\R$},
\end{align}
i.e., $W$ in \eqref{constrained_density} measures the amount of slip, so that we can put $W_{\rm el}=0$ on $\SO(2)$ and $W_{\rm pl}(F_{\rm pl})=\gamma^2$ for $F_{\rm pl} = \Id+\gamma s\otimes m$ in \eqref{eq:energydensity}.
Here, we are interested in the limit behavior of two-dimensional elastoplastic structures if the thickness in the $y_2$ direction tends to zero. 
More precisely, after a common thin structure rescaling, i.e., setting 
\begin{align}\label{rescaling}
	u(x) = v(y) \text{ with } x=(x_1,x_2) = (y_1,\tfrac{1}{h}y_2),
\end{align}
we obtain the energies per unit volume
\begin{align}\label{constrained_energies}
	\Ical_{h}: W^{1,2}(\Omega;\R^2) \to [0,\infty],\: u\mapsto \int_{\Omega} W(\nabla^h u) \dd x,
\end{align}
where $\nabla^h u = (\partial_1 u,\tfrac{1}{h}\partial_2 u)$ denotes the rescaled gradient and $\Omega :=\Omega_1$.
The $\Gamma$-limit of the sequence of rescaled energies is described in our following main result. 

\begin{theorem}\label{theo:hard_constraint}
	The sequence $(\Ical_h)_h$ as in \eqref{constrained_energies}, see also \eqref{constrained_density}, $\Gamma$-converges with respect to the weak topology in $W^{1,2}(\Omega;\R^2)$ to
	\begin{align*}
		\Ical: W^{1,2}(\Omega;\R^2) \to [0,\infty],\: u\mapsto \begin{cases}
																2\displaystyle\int_0^L \Wmin^{\rm c}(u') \dd x_1, &\text{ if } u\in  \Acal,\\
																\infty &\text{ otherwise,}
															\end{cases}
	\end{align*}
	with $\Acal:=\{u\in W^{1,2}(\Omega;\R^2) : \partial_2 u = 0\}$. The function $\Wmin : \R^3 \to [0,\infty]$ is given by
	\begin{align}\label{minW}
		\Wmin(\xi) = \min_{d\in\R^{2}} W(\xi|d),
	\end{align}
	and $(\cdot)^c$ stands for the convex envelope.
	
	Moreover, any sequence $(u_h)_h\subset W^{1,2}(\Omega;\R^2)$ with vanishing mean value and bounded energy, i.e., 
	$\int_\Omega u_h \dd x=0$ and $\Ical_{h}(u_{h})<C$ for a constant $C>0$ and all $k\in\N$, has a subsequence (not relabeled) 
	such that $u_h\weakly u$ in $W^{1,2}(\Omega;\R^2)$ for $u\in \Acal$ as $h\to 0$. 
	Additionally, if $s=\pm e_1$, then $u$ satisfies $|u'|\leq 1$ a.e.~in $(0,L)$.
\end{theorem}

The next corollary concerns the soft-constraint case, where we do not require rigid elasticity but allow a diverging elastic energy contribution as in \cite{CDK11}.

\begin{corollary}\label{cor:soft_constraint}
	For $\eps>0$ let
	\begin{align}\label{approx_densities}
		W_\eps: \R^{2\times 2}\to [0,\infty),\: F\mapsto \inf_{\gamma\in\R} \big(\tfrac{1}{\eps}\dist^2(F(\Id - \gamma s\otimes m),\SO(2)) + \gamma^2\big)
	\end{align}
	and consider the (rescaled) penalized energy
	\begin{align*}
		\Ical_{\eps,h}: W^{1,1}(\Omega;\R^2) \to [0,\infty),\: u\mapsto \int_{\Omega} W_\eps(\nabla^h v) \dd x;
	\end{align*}
	moreover, for any sequence $(\eps_k,h_k)_k$ with $(\eps_k,h_k)\to (0,0)$ we set $\Ical_k:=\Ical_{\eps_k,h_k}$.
	Then, the sequence $(\Ical_k)_k$ $\Gamma$-converges with respect to the weak topology in $W^{1,1}(\Omega;\R^3)$ to
	\begin{align*}
		\Ical: W^{1,1}(\Omega;\R^2) \to [0,\infty],\: u\mapsto \begin{cases}
																2\displaystyle\int_0^L \Wmin^{\rm c}(u') \dd x_1, &\text{ if } u\in \Acal,\\
																\infty &\text{ otherwise.}
															\end{cases}
	\end{align*}
	
	Moreover, any sequence $(u_k)_k\subset W^{1,1}(\Omega;\R^2)$ with vanishing mean value and bounded energy has a subsequence (not relabeled) 
	such that $u_k\weakly u$ in $W^{1,1}(\Omega;\R^2)$ for $u\in \Acal$ as $k\to \infty$. 
	Additionally, if $s=\pm e_1$, then $u$ satisfies $|u'|\leq 1$ a.e.~in $(0,L)$.
\end{corollary}

\begin{remark}
	 a) In the same spirit as in \cite{CDK11}, we can generalize our result to the case where we replace the elastic term $W_{\rm el}(F_{\rm el}) = \frac{1}{\eps}\dist^2(F_{\rm el},\SO(2))$ in \eqref{approx_densities} by one that satisfies
	\begin{align*}
		c\dist^2(F_{\rm el},\SO(2)) \leq \eps W_{\rm el}(F_{\rm el}) \leq C\dist^2(F_{\rm el},\SO(2))
	\end{align*}
	for constants $c,C>0$.
	
	\medskip
	
	b) The density $W_\eps$ as in \eqref{approx_densities} has linear growth, cf.~\cite[Equation (1.4)]{CDK11}. 
	Therefore, the corresponding energy functional is defined on $W^{1,1}(\Omega;\R^2)$ and the $\Gamma$-convergence result in Corollary \ref{cor:soft_constraint} is formulated with respect to the weak topology in $W^{1,1}(\Omega;\R^2)$ in contrast to the $W^{1,2}(\Omega;\R^2)$-setting in Theorem \ref{theo:hard_constraint}.

	\medskip	
	
	c) In Lemma \ref{lem:Wmin}, we compute an explicit expression of the limit density $\Wmin^{\rm c}$ depending on the slip direction $s$. 
	We find that $\Wmin^{\rm c}$ has quadratic growth and coercivity if $s\neq \pm e_1$, while the cases $s=\pm e_1$ yield a trivial density with either zero or an infinite energy contribution.
\end{remark}

This paper is organized as follows. 
After introducing some notation, we provide the reader in Section \ref{sec:basic} with some basic about the limit density and rank-one compatibility within $\Mcal_s$. 
In Section \ref{sec:preliminaries}, we establish the key lemmas for the construction of piecewise affine recovery sequences for the hard-constraint case of Theorem \ref{theo:hard_constraint}.
The proof of this main theorem is addressed in Section \ref{sec:hard_constraint}; the soft-constraint version, Corollary \ref{cor:soft_constraint}, is handled in Section \ref{sec:soft_constraint}.
We finish this paper in Section \ref{sec:gap} with a discussion why the recovery sequences for the hard-constraint case (at least for $s=\pm e_1)$ are required be non-smooth with jumps in the derivative.

\section{Preliminaries}\label{sec:preliminaries}
\subsection{Notation}
We recall some notation used throughout the article. The vectors $e_1,\ldots,e_n\in\R^n$ for $n\in\N$ denote the standard basis vectors in $\R^n$. 
We denote the Euclidean (and Frobenius) norm on $\R^n$ (and $\R^{n\times n}$) by $|\cdot|$, and $\Scal^{n-1}$ is the unit sphere in $\R^n$ centered at zero.
The set of proper rotations on $\R^n$ is defined as 
$$\SO(2)=\{A\in\R^{n\times n}:\, A^\top A=AA^\top=\Id,\, \det A =1\}.$$ 
For an angle $\ffi\in \R$ we write $R_\ffi := \begin{pmatrix}\cos\ffi & -\sin \ffi \\ \sin \ffi & \cos \ffi\end{pmatrix}\in\SO(2),$ and set $a^\perp=R_{\frac{\pi}{2}}a$ for $a\in\R^2$.
Let $\T^1$ denote the one-dimensional flat torus, i.e., the interval $[0,2\pi]$ with the end points identified, reflecting $2\pi$-periodicity.
For $\theta\in \T^1$ and $\gamma\in\R$, define 
\begin{align*}
	\RS(\theta,\gamma)=\RS(\theta,\gamma;s):=R_{\theta}(\Id + \gamma s\otimes m)\in\Mcal_s.
\end{align*}
Notice that for any $F\in \Mcal_s$ there exist uniquely defined $(\theta,\gamma)\in \T^1\times \R$ such that $F=\RS(\theta,\gamma)$.
The function $\mathbbm{1}_E$ denotes the indicator function of a set $E\subset \R^m$, which equals $1$ on $E$ while vanishing elsewhere. 
Given a function $W:\R^n\to\R$, $W^c$ denotes its convex envelope, i.e., the pointwise supremum of all affine functions not greater than $W$. 
The set 
\begin{align*}
	\Acal:=\{u\in W^{1,2}(\Omega;\R^2) : \partial_2 u = 0\}
\end{align*}
 is the subspace of all Sobolev functions on $\Omega = (0,L)\times (-1,1)$ that are constant with respect to $x_2$. This set can alternatively be identified with $W^{1,2}((0,L);\R^2)$.

In the rest of this section, we collect some auxiliary results needed for the proofs of the main results in Section~\ref{sec:proofs}. 

\subsection{A few basic observations}\label{sec:basic}

First, we provide an explicit expression for $\Wmin^{\rm c}$, cf.~\eqref{minW}.
\begin{lemma}\label{lem:Wmin}
	Let $s,m\in\Scal^1$ with $m=s^\perp$ and $\Wmin$ as in \eqref{minW}.
	
	a) If $s=\pm e_1$, then
	\begin{align*}
		\Wmin^c(\xi) = \begin{cases}
						\infty &\text{ if } |\xi|>1,\\
						0 &\text{ if } |\xi|\leq 1,
					\end{cases},
	\end{align*}
	for all $\xi \in \R^2$.
	
	b) In the case $s \neq \pm e_1$, it holds that
	\begin{align*}
		\Wmin^{\rm c}(\xi) = \begin{cases}
						\frac{1}{s_2^2}|\xi|^2 - 2\frac{|s_1|}{s_2^2}\sqrt{|\xi|^2 - s_2^2} + \frac{s_1^2}{s_2^2} -1 &\text{ if } |\xi|>1,\\
						0 &\text{ if } |\xi|\leq 1,
					\end{cases}
	\end{align*}
	for all $\xi\in\R^2$; in particular $\Wmin^{\rm c}$ satisfies
	\begin{align}\label{Wmin_growth}
		c|\xi|^2 - C \leq \Wmin^{\rm c}(\xi) \leq C(1+|\xi|^2)
	\end{align}
	for some constants $c,C>0$.
	Moreover, for every $\xi\in \R^2$ with $|\xi|\geq 1$, there exist  $\theta\in \T^1$ and $\gamma\in\R$ such that
	\begin{align*}
		\xi = \RS(\theta,\gamma)e_1 \qand \Wmin^c(\xi) = \Wmin(\xi) = \gamma^2.
	\end{align*}	
\end{lemma}
\begin{proof}
	Let $\xi\in \R^2$, then $\xi$ is the first column of an element in $\Mcal_{s}$ if and only if there exists $\gamma\in\R$ and $\theta\in \T^1$ such that
	$\xi = \RS(\theta,\gamma)e_1 = R_\theta(\Id + \gamma s\otimes m)e_1$. This is equivalent to the existence of $\gamma\in \R$ with
	\begin{align}\label{first_column}
		|\xi|^2 =|(\Id + \gamma s\otimes m)e_1|^2=\Big|\begin{pmatrix}1+\gamma s_1m_1\\\gamma s_2m_1\end{pmatrix}\Big|^2 = 1 + 2\gamma s_1m_1 + \gamma^2m_1^2 = 1 - 2\gamma s_1s_2 + \gamma^2 s_2^2.
	\end{align}
	The energy contribution associated to this shear is then exactly
	\begin{align}\label{shear_energy}
		W(\Id + \gamma s\otimes m) = \gamma^2.
	\end{align}
	
	a) If $s=\pm e_1$, then \eqref{first_column} can only be satisfied if $|\xi|=1$. In this case, we choose $\gamma=0$ to minimize \eqref{shear_energy}.
	A subsequent convexification of the energy yields the desired result.	

	b) Now, let $s\neq \pm e_1$. 
	Analogously to the case $s=\pm e_1$, it suffices to consider the case $|\xi|>1$, since $|\xi|=1$ yields a vanishing energy contribution with $\gamma=0$. 
	Solving \eqref{first_column} for $\gamma$ produces the solutions $\gamma_\pm$ with
	\begin{align*}
		\gamma_\pm^2 = \frac{1}{s_2^2}|\xi|^2 \pm 2\frac{s_1}{s_2^2}\sqrt{|\xi|^2 - s_2^2} + \frac{s_1^2}{s_2^2} - 1
	\end{align*}
	After selecting the solution with smaller magnitude, we obtain the formula in question.
	The quadratic growth and coercivity is apparent from the explicit formula.
\end{proof}

The next Lemma about rank-one connections within $\Mcal_s$ is extracted from \cite[Lemma 3.1]{ChK17} and is essential in the construction of piecewise affine functions whose gradients are contained in $\Mcal_s$.
The proof is essentially based on the case $s=e_1$ considering that $\Mcal_{s} = \Mcal_{e_1}S^T$ with $S=(s|m)\in\SO(2)$.
\begin{lemma}\label{lem:compatibility}
	Let $s\in\Scal^1$, $\theta_1,\theta_2\in \T^1$, and $\gamma_1,\gamma_2\in\R$. Then $\RS(\theta_1,\gamma_1)$ and $\RS(\theta_2,\gamma_2)$ are rank-one connected if and only if one of the following two conditions holds:
	
	$i)$ $\theta_2 = \theta_1$ and $\gamma_1\neq \gamma_2$; in this case,
	\begin{align*}
		\RS(\theta_2,\gamma_2) - \RS(\theta_1,\gamma_1) =  (\gamma_2 - \gamma_1)R_{\theta_1}s\otimes m.
	\end{align*}
	
	$ii)$ $\theta_1\neq\theta_2$ and $\gamma_2 - \gamma_1 = 2 \tan\big(\frac{\theta}{2}\big)$ with $\theta = \theta_2 - \theta_1 \in (-\pi,\pi)$; in this case,
	\begin{align*}
		\RS(\theta_2,\gamma_2) - \RS(\theta_1,\gamma_1) = \frac{\gamma_2 - \gamma_1}{4 + (\gamma_2 - \gamma_1)^2}R_\theta\big((\gamma_2-\gamma_1)s + 2m\big)\otimes\big(2s + (\gamma_2 + \gamma_1)m\big)
	\end{align*}
\end{lemma}

\subsection{Building blocks for the construction of recovery sequences}\label{sec:building_blocks}

Below, we present a collection of lemmas intended as building blocks for the construction of a piecewise defined recovery sequence.
For the first few elementary lemmas, the general philosophy is the following: 
We start with a given admissible deformation whose gradient coincides with fixed, prescribed rotated shears at the both ends (i.e., for $x_1\gg 0$ and $x_1\ll 0$), and then modify it to another of essentially the same kind, but with changed parameters for the rotated shears at the two ends.
Ultimately, we intend to glue many such pieces to construct suitable recovery sequences, in such a way that wherever two pieces overlap, they have identical gradient, so that a continuous transition is easily achieved by adding appropriate constants if necessary. 
The crucial elementary building block is the one that allows us to form a kink without violating the hard constraint, presented in Lemma~\ref{lem:angle_e1e2} for the cases of $s\in \{\pm e_1,\pm e_2\}$, to be generalized later with the help of the other lemmas. 
The joint result of the constructions is summarized in Lemma~\ref{lem:switch_rotated_shear}, which then is used at the heart of the proof of Theorem~\ref{theo:hard_constraint}.

\begin{lemma}[Change $s$ globally]\label{lem:changes}
Let $s,\hat{s}\in\Scal^1$ with $s\cdot \hat{s}\geq 2^{-\frac{1}{2}}$ (i.e., forming an angle of at most $\frac{\pi}{4}$), and $B>0$. 
Moreover, let $w\in W_\loc^{1,\infty}(\R\times (-B,B);\R^2)$ such that $\nabla w\in \Mcal_s$ a.e.~and
\begin{align}\label{BBBC3}
	\nabla w(x_1,\cdot)=
	\begin{cases}
		\RS(\theta_1,\gamma_1;s)&~\text{for $x_1<-\tfrac{1}{2}B$}.\\
		\RS(\theta_2,\gamma_2;s)&~\text{for $x_1>\tfrac{1}{2}B$.}
	\end{cases}
\end{align}
Then, there exists a function
$v\in W_\loc^{1,\infty}(\R\times (-\tfrac{1}{8}B,\tfrac{1}{8}B);\R^2)$ such that $\nabla v\in \Mcal_{\hat{s}}$ a.e.~and
\begin{align}\label{BBBCs}
	\nabla v(x_1,\cdot)=
	\begin{cases}
		\RS(\theta_1,\gamma_1;\hat{s})&~\text{for $x_1<-\tfrac{7}{8}B$},\\
		\RS(\theta_2,\gamma_2;\hat{s})&~\text{for $x_1>\tfrac{7}{8} B$.}
	\end{cases}
\end{align}
\end{lemma}
\begin{proof}
Since $s\cdot \hat{s}\geq \frac{1}{\sqrt{2}}$, there exists $\theta\in [-\pi/4,\pi/4]$
such that $\hat{s}=R_\theta s$.
We abbreviate $b:=\frac{1}{8}B$ and define $v(x):=R_{\theta} w(R_{-\theta} x)$ for $\abs{x_1}\leq 7b$ and $\abs{x_2}<b$, 
whence
\[
	\nabla v(x)=R_{\theta} (\nabla w)(R_{-\theta} x) R_{-\theta} \in \Mcal_{\hat{s}},
\]
the latter because
$(\Id+\gamma s\otimes m)R_{-\theta}=R_{-\theta}+\gamma s\otimes (R_\theta m)=R_{-\theta} (\Id+\gamma \hat{s}\otimes \hat{m})$ for all $\gamma\in\R$.
If we continuously extend $v$ by affine functions on both sides for $\abs{x_1}>7b$, 
with constant gradients given by \eqref{BBBCs},
the function $v$ is well defined on $\R\times (-b,b)$ as long as $R_{-\theta}$ maps the two borderlines in \eqref{BBBCs} to sets where $\nabla w$ is fixed by \eqref{BBBC3},
i.e., if
\begin{align*}
	&\text{$R_{-\theta}$ maps $\{7 b\}\times (-b,b)$ 
into $\big(\tfrac{1}{2}B,\infty\big)\times (-B,B)$, and}\\
	&\text{$R_{-\theta}$ maps $\{-7 b\}\times (-b,b)$ 
into $\big(-\infty,-\tfrac{1}{2}B\big)\times (-B,B)$,}
\end{align*}
see also Figure \ref{fig:changes}.
Since $7^2 b^2+b^2< B^2$, it is clear that $R_{-\theta}(\pm 7 b,\pm b)^\top\in (-B,B)^2$. 
It remains to show that $\abs{e_1\cdot R_{-\theta}(\pm 7 b,\pm b)}> \frac{1}{2}B$.
This does hold for all $\abs{\theta}\leq \frac{\pi}{4}$ since
\[
	\cos(\abs{\theta})7 b-\sin(\abs{\theta}) b 
\geq \frac{1}{\sqrt{2}}7 b-\frac{1}{\sqrt{2}} b=
\frac{1}{\sqrt{2}}\frac{3}{4}B > \frac{1}{2}B.
\] 
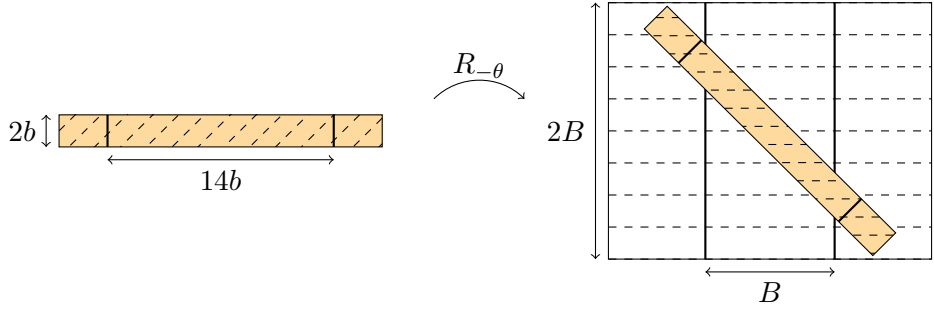
\begin{figure}
	\centering
	\begin{tikzpicture}[scale=.85]
		\draw [fill=Dandelion!50!white] (-2.5,-.25) --++ (5,0) --++ (0,.5) --++ (-5,0) -- cycle;
		\draw [thick] (-1.75,-.25) --++ (0,.5);
		\draw [thick] (1.75,-.25) --++ (0,.5);
		\draw [<->] (-1.75,-0.45) --++ (1.75,0)  node [anchor = north] {$14b$} --++(1.75,0);
		\draw [<->] (-2.7,-.25) --++ (0,.25)  node [anchor = east] {$2b$} --++(0,.25);
		\begin{scope}
			\clip (-2.5,-.25) --++ (5,0) --++ (0,.5) --++ (-5,0) -- cycle;
			\foreach \x in {-15,...,5}
			    \draw[dashed] (-2.5,0.4*\x) --++ (45:8);
		\end{scope}

		\draw [->] (3.3,.5) to [out=45,in=135] (4.7,.5);
		\draw (4,1) node {$R_{-\theta}$};		
			
		\begin{scope}[shift={(8.5,0)}]
			\draw (-2.5,-2) --++ (5,0) --++ (0,4) --++ (-5,0) -- cycle;
			\draw [thick] (-1,-2) --++ (0,4);
			\draw [thick] (1,-2) --++ (0,4);
			\draw [<->] (-1,-2.2) --++ (1,0)  node [anchor = north] {$B$} --++(1,0);
			\draw [<->] (-2.7,-2) --++ (0,2)  node [anchor = east] {$2B$} --++(0,2);
			\foreach \x in {-4,...,4}
		    \draw[dashed] (-2.5,0.5*\x) --++ (5,0);
		    
		    \begin{scope}[shift={(0,0)},rotate=-45]
		    	\draw [fill=Dandelion!50!white] (-2.5,-.25) --++ (5,0) --++ (0,.5) --++ (-5,0) -- cycle;
				\draw [thick] (-1.75,-.25) --++ (0,.5);
				\draw [thick] (1.75,-.25) --++ (0,.5);
				\begin{scope}
					\clip (-2.5,-.25) --++ (5,0) --++ (0,.5) --++ (-5,0) -- cycle;
					\foreach \x in {-15,...,5}
					    \draw[dashed] (-2.5,0.4*\x) --++ (45:8);
				\end{scope}
		    \end{scope}
		\end{scope}
	\end{tikzpicture}
	\caption{A part of the domain $\R\times (-b,b)$ of $v$ relevant for the boundary conditions in \eqref{BBBCs} for $\hat s=\frac{1}{\sqrt{2}}(e_1+e_2)$ and its image under $R_{-\theta}$ for $\theta=\frac{\pi}{4}$ embedded into the larger domain $\R\times (-B,B)$ of $w$ with boundary conditions \eqref{BBBC3} for $s = e_1$. The dashed lines describe the (deformed) slip directions.}\label{fig:changes}
\end{figure}
\end{proof}

\begin{lemma}[Change shear for $s\neq \pm e_1$]\label{lem:changegamma}
Let $s\in \Scal^1\setminus \{\pm e_1\}$, $\theta_1,\theta_2\in \T^1$, $\gamma_1,\gamma_2\in \R$ and $A,B>0$.
Moreover, let $w\in W^{1,\infty}(\R\times (-B,B);\R^2)$ such that $\nabla w\in \Mcal_s$ a.e.~and
\begin{align*}
	\nabla w(x_1,\cdot)=
	\begin{cases}
		\RS(\theta_1,\gamma_1;s)&~\text{for $x_1<-A$},\\
		\RS(\theta_2,\gamma_2;s)&~\text{for $x_1>A$.}
	\end{cases}
\end{align*}
Then for all $\tilde{\gamma}_1,\tilde{\gamma}_2\in\R$, there exists $v\in W_\loc^{1,\infty}(\R\times (-B,B);\R^2)$ such that $\nabla v\in \Mcal_s$ a.e.~and
\begin{align}\label{BBBCgammas}
	\nabla v(x_1,\cdot)=
	\begin{cases}
		\RS(\theta_1,\tilde{\gamma}_1;s)&~\text{for $x_1<-a$},\\
		\RS(\theta_2,\tilde{\gamma}_2;s)&~\text{for $x_1>a$},
	\end{cases}
	\quad \text{where}~a:=\frac{A+\abs{m_2}B}{\abs{m_1}}>A.
\end{align}
\end{lemma}
\begin{proof}
Let $\sigma:=\frac{m_1}{|m_1|}\in \{-1,1\}$; this is a fixed sign chosen so that $\sigma m_1>0$.
We define $\mu>0$ and sets $S_i\subset \R\times (-B,B)$ by
\[
\begin{alignedat}{2}
	&\mu&&:=A+\abs{m_2}B>0\\
	&S_1&&:=\left\{ x\in \R^2 \,\left|\, x\cdot (\sigma m) \leq -\mu  \right.\right\},\\
	&S_2&&:=\left\{ x\in \R^2 \,\left|\, x\cdot (\sigma m) \geq \mu \right.\right\},\\
	&S_0&&:=\R^2\setminus (S_1\cup S_2).\\
\end{alignedat}
\]
As defined, $\Gamma_j:=\partial S_j\cap (\R\times (-B,B))$, $j=1,2$, are lines parallel to $s$.
Moreover, by the definition of $\mu$, $\Gamma_1$ and $\Gamma_2$ do not intersect
$(-A,A)\times (-B,B)$, only touching one corner while passing it on the left or right: 
For any $x_2\in (-B,B)$ and $x_1$ such that $x=(x_1,x_2)\in \Gamma_2$ (the case of $\Gamma_1$ is analogous), we have that
$x_1\sigma m_1=\mu-\sigma m_2x_2>\mu-\abs{m_2}B=A$.
In particular, $(-A,A)\times (-B,B)\subset S_0$.

Now let $v:=w$ on $S_0\cap (\R\times (-B,B))$, and continuously extend it to $\R\times (-B,B)$ by two affine functions 
on $S_1$ and $S_2$, respectively, with fixed gradients given by the two matrices of \eqref{BBBCgammas}.
Since $\partial S_0\cap [\R\times (-B,B)]=\Gamma_1\cup \Gamma_2$,
this gives a well defined function in $W^{1,\infty}(\R\times (-B,B);\R^2)$ with $\nabla v\in \Mcal_s$ a.e.. 
Here, recall that by Lemma~\ref{lem:compatibility} (i),
$\RS(\theta_j,\tilde{\gamma}_j)$ and $\RS(\theta_j,\gamma_j)$ 
are rank-one connected with normal $m$ and therefore compatible across the interfaces $\Gamma_j$.
Finally, we also have
\eqref{BBBCgammas} since $a=\sup \{|x\cdot e_1| : x\in S_0\}$.
\end{proof}

\begin{lemma}[Change angle globally]\label{lem:changetheta}
Let $s\in\Scal^1$, $\theta_1,\theta_2,\theta\in \T^1$, $\gamma_1,\gamma_2\in \R$ and $A\geq B>0$. 
Moreover, let $w\in W_\loc^{1,\infty}(\R\times (-B,B);\R^2)$ such that $\nabla w\in \Mcal_s$ a.e.~and
\begin{align*}
	\nabla w(x_1,\cdot)=\begin{cases} 
		\RS(\theta_1,\gamma_1)&~\text{for $x_1<-A$},\\
		\RS(\theta_2,\gamma_2)&~\text{for $x_1>A$.}
	\end{cases}
\end{align*}
Then, there exists $v\in W_\loc^{1,\infty}(\R\times (-B,B);\R^2)$ such that $\nabla v\in \Mcal_s$ a.e.,
\begin{align*}
	\nabla v(x_1,\cdot)=\begin{cases}
		\RS(\theta_1+\theta,\gamma_1)&~\text{for $x_1<-A$},\\
		\RS(\theta_2+\theta,\gamma_2)&~\text{for $x_1>A$.}
	\end{cases}
\end{align*}
\end{lemma}
\begin{proof}
The function $v:=R_\theta w$ has the asserted properties.
\end{proof}

As we have seen in Lemma \ref{lem:changes}, we can only transform maps whose gradients lie in $\Mcal_{s}$ for some $s\in\Scal^1$ to those with gradients in $\Mcal_{\hat{s}}$ for $\hat{s}\in\Scal^1$ if $s$ and $\hat{s}$ form an angle of at most $\frac{\pi}{4}$.
This is why it is necessary and sufficient to explicitly construct building blocks only for the slip directions $s=\pm e_1$ and $s=\pm e_2$.

We first present a piecewise affine construction that allows a (small) change of angle on one side of the domain while keeping the shear fixed on both sides.

\begin{lemma}[Changing angles for $s\in\{\pm e_1,\pm e_2\}$]\label{lem:angle_e1e2}
Let $B>0$ and $s\in\{\pm e_1,\pm e_2\}$. There exists $\theta_{\max}\in (0,\pi]$ (independent of $B$) such that for every $|\theta|\leq \theta_{\max}$ there is $w\in W_\loc^{1,\infty}(\R\times(-B,B);\R^2)$ satisfying $\nabla w\in \Mcal_{s}$ almost everywhere and
\begin{align}\label{w_construction}
	\nabla w(x_1,\cdot)=\begin{cases}
		\Id &\text{ for }x_1<-\tfrac{1}{2}B,\\
		 R_{\theta} &\text{ for } x_1>\tfrac{1}{2}B.		
	\end{cases}
\end{align}
\end{lemma}
\begin{proof}
	The case $\theta = 0$ is trivial.
	In the following, we will split the domain $E=\R\times(-B,B)$ into four pieces:
	\begin{align}\label{partition}
		\begin{split}
			E_1 &= \{(x_1,x_2)\in E: x_2\leq -x_1\cot\ffi,\, x_1\leq 0\},\quad E_2=\{(x_1,x_2)\in E : -x_1\cot \ffi <x_2,\, x_1\leq 0\}\\
			E_3 &= \{(x_1,x_2)\in E: x_2\geq x_1\cot \ffi,\, x_1>0\},\quad E_4=\{(x_1,x_2)\in E : x_2< x_1\cot \ffi,\, x_1>0\},
		\end{split}
	\end{align}
	where $\ffi\in(0,\arctan(\frac{1}{4})]$. The upper bound comes from $\tan \ffi \leq \frac{\frac{B}{2}}{2B} = \frac{1}{4}$, see the boundary conditions in \eqref{w_construction}.
	
	\textit{Part 1: The case $s=\pm e_1$.} 
	Let us first assume that $\theta>0$. 
	Considering the partition \eqref{partition}, we define a piecewise affine function $w_\ffi: E \to \R^2$ with gradients
	\begin{align}\label{w_ffi}
		\nabla w_\ffi(x) =
			\begin{cases}
				\Id &\text{ if } x\in E_1,\\
				\RS(\omega_2,2\tan\ffi;\pm e_1) &\text{ if } x\in E_2,\\
				\RS(\omega_3,-2\tan\ffi;\pm e_1) &\text{ if } x\in E_3,\\
				R_{\omega_4} &\text{ if } x\in E_4,
			\end{cases}\quad x\in E
	\end{align}
	and
	\begin{align*}
		\omega_2 &= 2\arctan(\tan\ffi) = 2\ffi, \\
		\omega_3 &= -2\arctan(2\tan \ffi) + \omega_2 = -2\arctan(2\tan\ffi) + 2\ffi,\\
		\omega_4 &= 2\arctan \tan \ffi + \omega_3 = 4 \ffi - 2\arctan(2\tan\ffi),
	\end{align*}
	see Figure \ref{fig:bend1}.
	Lemma \ref{lem:compatibility} $ii)$ immediately yields that this deformation $w_\ffi$ is continuous since the rank-one compatibilities along all appearing interfaces are satisfied.
	Moreover, we observe that $\omega_4$ is continuous, increasing in $\ffi$, and satisfies
	\begin{align*}
		\lim_{\ffi\to 0} \omega(\ffi) = 0\qand \lim_{\ffi\to \frac{\pi}{2}} \omega_4(\ffi) = 4\frac{\pi}{2}- 2\frac{\pi}{2} = \pi.
	\end{align*}
	We now set $\theta_{\max} = \omega_4(\arctan(\frac{1}{4}))$ and find for any given $\theta\in (0,\theta_{\max}]$ some $\bar{\ffi}\in (0,\arctan(\frac{1}{4})]$ such that $\theta = \omega_4(\bar{\ffi})$.
	The desired map $w$ is then given by $w_{\bar{\ffi}}$ as in \eqref{w_ffi}, see also Figure \ref{fig:bend1}.
	
	If $\theta<0$, then we mirror the previous construction.
	Precisely, we set $A=e_1\otimes e_1 - e_2\otimes e_2$ and define $w(x) = A^T w_{\bar{\ffi}}(Ax)$ for every $x\in E$. 
	We then observe that $w$ satisfies the desired boundary values and that $\nabla w(x)\in\Mcal_{e_1}$ since
	\begin{align*}
		A^T\RS(\theta,\gamma;\pm e_1)A= \RS(-\theta,-\gamma;\pm e_1)
	\end{align*}
	for any $\theta,\gamma\in \R$.
	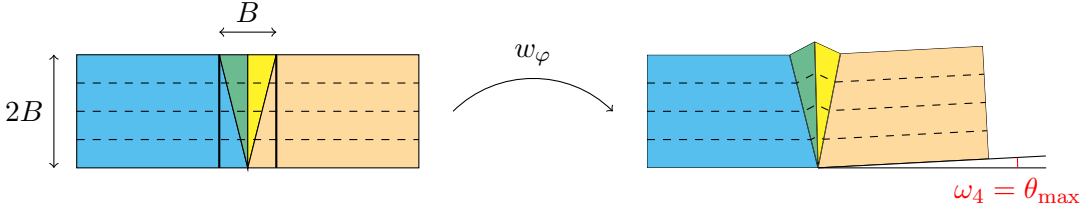
\begin{figure}
		\centering
		\begin{tikzpicture}[scale=1.5]
			\def\angle{14.04}
			
			\pgfmathsetmacro\gammab{2*tan(\angle)}
			\pgfmathsetmacro\gammac{2*tan(0)-\gammab}
			\pgfmathsetmacro\gammad{-2*tan(\angle)-\gammac}
		
			\pgfmathsetmacro\thetab{2*\angle}
			\pgfmathsetmacro\thetac{-2*atan(2*tan(\angle)) + 2*\angle}
			\pgfmathsetmacro\thetad{4*\angle -2*atan(2*tan(\angle))}

			\draw [fill=Cerulean!60!white] (-1.5,0) --++ (1.5,0) --++ ($(0,1) - ({tan(\angle)},0)$) -- (-1.5,1) -- cycle;
			\draw [fill=ForestGreen!60!white] (0,0) -- ($(0,1) - ({tan(\angle)},0)$) -- (0,1) -- cycle;
			\draw [fill=yellow!90!white] (0,0) -- (0,1) -- ($(0,1) + ({tan(\angle)},0)$) -- cycle;
			\draw [fill=Dandelion!50!white] (0,0) -- ($(0,1) + ({tan(\angle)},0)$) -- (1.5,1) --++ (0,-1) -- cycle;
			\foreach \x in {1,...,3}
		    	\draw[dashed] (-1.5,0.25*\x) --++ (3,0);
			\draw [<->] (-1.7,0) --++ (0,.5)  node [anchor = east] {$2B$} --++(0,.5);
			\draw [<->] (-.25,1.2) --++ (.25,0) node [anchor = south] {$B$} --++ (.25,0);
			\draw [thick] (-.25,0) --++ (0,1);
			\draw [thick] (.25,0) --++ (0,1);
		
			\draw [->] (1.8,.5) to [out=45,in=135] (3.2,.5);
			\draw (2.5,1) node {$w_\ffi$};		
			
			\begin{scope}[shift={(5,0)}]
				\begin{scope}
					\clip (-1.5,0) --++ (1.5,0) --++ ($(0,1) - ({tan(\angle)},0)$) -- (-1.5,1) -- cycle;
					\draw [fill=Cerulean!60!white] (-1.5,0) --++ (1.5,0) --++ ($(0,1) - ({tan(\angle)},0)$) -- (-1.5,1) -- cycle;
					\foreach \x in {1,...,3}
		    			\draw[dashed] (-1.5,0.25*\x) --++ (3,0);
				\end{scope}	
				
				\begin{scope}[ rotate={\thetab}, xslant=\gammab]
					\clip (0,0) -- ($(0,1) - ({tan(\angle)},0)$) -- (0,1) -- cycle;
					\draw [fill=ForestGreen!60!white] (0,0) -- ($(0,1) - ({tan(\angle)},0)$) -- (0,1) -- cycle;
					\foreach \x in {1,...,3}
		    			\draw[dashed] (-1.5,0.25*\x) --++ (3,0);
				\end{scope}
				
				\begin{scope}[ rotate=\thetac, xslant=\gammac]
					\clip (0,0) -- (0,1) -- ($(0,1) + ({tan(\angle)},0)$) -- cycle;
					\draw [fill=yellow!90!white] (0,0) -- (0,1) -- ($(0,1) + ({tan(\angle)},0)$) -- cycle;
					\foreach \x in {1,...,3}
		    			\draw[dashed] (-1.5,0.25*\x) --++ (3,0);
				\end{scope}
				
				\begin{scope}[ rotate=\thetad, xslant=\gammad]
					\clip (0,0) -- ($(0,1) + ({tan(\angle)},0)$) -- (1.5,1) --++ (0,-1) -- cycle;
					\draw [fill=Dandelion!50!white] (0,0) -- ($(0,1) + ({tan(\angle)},0)$) -- (1.5,1) --++ (0,-1) -- cycle;
					\foreach \x in {1,...,3}
		    			\draw[dashed] (-1.5,0.25*\x) --++ (3,0);
				\end{scope}
				\draw (0,0) -- (2,0);
				\draw (0,0) -- (3.015:2);
				\draw [red,domain=0:3.015] plot ({1.75*cos(\x)}, {1.75*sin(\x)});
				\draw (1.75,0) node [red, anchor=north] {$\omega_4=\theta_{\max}$};
			\end{scope}
		\end{tikzpicture}
		\caption{On the left: a part of the reference configuration $\R\times(-B,B)$ partitioned into the four colored subsets $E_1, E_2, E_3, E_4$ in \eqref{partition} for the maximal choice $\ffi = \arctan(\frac{1}{4})$; the dashed lines indicate the slip direction $s=e_1$. On the right: The image under the continuous piecewise affine map $w_\varphi$ as in \eqref{w_ffi}.}\label{fig:bend1}
	\end{figure}
	
	\textit{Part 2: The case $s=\pm e_2$.}	
	We start again with $\theta>0$.
	In light of Lemma \ref{lem:compatibility} $ii)$, we find that the piecewise affine function $w_\ffi: E \to \R^2$ with gradients
	\begin{align}\label{w_ffi_e_2}
		\nabla w_\ffi(x) =
			\begin{cases}
				\Id &\text{ if } x\in E_1,\\
				\RS(\omega_2,2\cot \ffi;\pm e_2) &\text{ if } x\in E_2,\\
				\RS(\omega_3,-2\cot \ffi;\pm e_2) &\text{ if } x\in E_3,\\
				R_{\omega_4} &\text{ if } x\in E_4,
			\end{cases}\quad x\in E
	\end{align}
	and
	\begin{align*}
		\omega_2 = 2\arctan(-\cot\ffi) =  2\ffi - \pi,\quad \omega_3= \omega_2, \quad \omega_4 = \arctan(-\cot \ffi) + \omega_3 = 4 \ffi - 2\pi
	\end{align*}
	is continuous. Note that we may replace the expression for $\omega_4$ by $\omega_4 = 4\ffi$ since $R_{4\ffi - 2\pi} = R_{4\ffi}$.
	Similarly to before, we set $\theta_{\max} = \omega_4(\arctan(\frac{1}{4})) = 4\arctan(\frac{1}{4})$ and set $\bar{\ffi} = \frac{1}{4}\theta \in (0,\arctan(\frac{1}{4})]$ for $\theta\in (0,\theta_{\max}]$.
	The desired map $w$ is then given by $w_{\bar{\ffi}}$ as in \eqref{w_ffi_e_2}, see also Figure \ref{fig:bend2}.
	If $\theta<0$, then we mirror this construction exactly as in Part 1 of this proof.
	
	\begin{figure}
		\centering
		\begin{tikzpicture}[scale=1.5]
			\def\angle{14.04}	
			\def\tb{0}	
			
			\pgfmathsetmacro\gammab{2*cot(\angle)}
			\pgfmathsetmacro\gammac{-2*cot(\angle)}
			
			\pgfmathsetmacro\thetab{2*\angle- 180}
			\pgfmathsetmacro\thetac{2*\angle - 180}
			\pgfmathsetmacro\thetad{4*\angle}

			\draw [fill=Cerulean!60!white] (-1.5,0) --++ (1.5,0) --++ ($(0,1) - ({tan(\angle)},0)$) -- (-1.5,1) -- cycle;
			\draw [fill=ForestGreen!60!white] (0,0) -- ($(0,1) - ({tan(\angle)},0)$) -- (0,1) -- cycle;
			\draw [fill=yellow!90!white] (0,0) -- (0,1) -- ($(0,1) + ({tan(\angle)},0)$) -- cycle;
			\draw [fill=Dandelion!50!white] (0,0) -- ($(0,1) + ({tan(\angle)},0)$) -- (1.5,1) --++ (0,-1) -- cycle;
			\begin{scope}
			\clip(-1.5,0)--++(3,0)--++(0,1)--++(-3,0)--cycle;			
				\foreach \x in {-7,...,7}
			    	\draw[dashed] (0.4*\x,0) --++ (0,1);
			\end{scope}
			\draw [<->] (-1.7,0) --++ (0,.5)  node [anchor = east] {$2B$} --++(0,.5);
			\draw [<->] (-.25,1.2) --++ (.25,0) node [anchor = south] {$B$} --++ (.25,0);
			\draw [thick] (-.25,0) --++ (0,1);
			\draw [thick] (.25,0) --++ (0,1);
		
			\draw [->] (1.8,.5) to [out=45,in=135] (3.2,.5);
			\draw (2.5,1) node {$w_\ffi$};		
			
			\begin{scope}[shift={(5,0)}]
				\begin{scope}[rotate=\thetad]
					\clip (0,0) -- ($(0,1) + ({tan(\angle)},0)$) -- (1.5,1) --++ (0,-1) -- cycle;
					\draw [fill=Dandelion!50!white] (0,0) -- ($(0,1) + ({tan(\angle)},0)$) -- (1.5,1) --++ (0,-1) -- cycle;
					\foreach \x in {-7,...,7}
			    		\draw[dashed] (0.4*\x,0) --++ (0,1);
				\end{scope}
				
				\begin{scope}[rotate={\thetab},yslant=\gammab]
					\clip (0,0) -- ($(0,1) - ({tan(\angle)},0)$) -- (0,1) -- cycle;
					\draw [fill=ForestGreen!60!white] (0,0) -- ($(0,1) - ({tan(\angle)},0)$) -- (0,1) -- cycle;
			    	\foreach \x in {-7,...,7}
			    		\draw[dashed] (0.4*\x,0) --++ (0,1);
				\end{scope}					
				
				\begin{scope}[rotate=\thetac, yslant=\gammac]
					\clip (0,0) -- (0,1) -- ($(0,1) + ({tan(\angle)},0)$) -- cycle;
					\draw [fill=yellow!90!white] (0,0) -- (0,1) -- ($(0,1) + ({tan(\angle)},0)$) -- cycle;
					\foreach \x in {-7,...,7}
			    		\draw[dashed] (0.4*\x,0) --++ (0,1);
				\end{scope}
				
				\begin{scope}
					\clip (-1.5,0) --++ (1.5,0) --++ ($(0,1) - ({tan(\angle)},0)$) -- (-1.5,1) -- cycle;
					\draw [fill=Cerulean!60!white] (-1.5,0) --++ (1.5,0) --++ ($(0,1) - ({tan(\angle)},0)$) -- (-1.5,1) -- cycle;
					\foreach \x in {-7,...,7}
			    		\draw[dashed] (0.4*\x,0) --++ (0,1);
				\end{scope}
				\draw (0,0) -- (1,0);
				\draw [red,domain=0:56.14] plot ({.75*cos(\x)}, {.75*sin(\x)});
				\draw ($({.75*cos(28.07)}, {.75*sin(28.07)})$) node [red, anchor=west] {$\omega_4=\theta_{\max}$};

				\clip (-1.5,0) --++ (3,0) --++ (0,1) --++ (-3,0) -- cycle;
				\begin{scope}[rotate=\thetad]
					\clip (0,0) -- ($(0,1) + ({tan(\angle)},0)$) -- (1.5,1) --++ (0,-1) -- cycle;
					\draw [fill=Dandelion!50!white,fill opacity=0.5] (0,0) -- ($(0,1) + ({tan(\angle)},0)$) -- (1.5,1) --++ (0,-1) -- cycle;
					\foreach \x in {-7,...,7}
			    		\draw[dashed] (0.4*\x,0) --++ (0,1);
				\end{scope}
				
				\begin{scope}[rotate=\thetac, yslant=\gammac]
					\clip (0,0) -- (0,1) -- ($(0,1) + ({tan(\angle)},0)$) -- cycle;
					\draw [fill=yellow!90!white,fill opacity=0.5] (0,0) -- (0,1) -- ($(0,1) + ({tan(\angle)},0)$) -- cycle;
					\foreach \x in {-7,...,7}
			    		\draw[dashed] (0.4*\x,0) --++ (0,1);
				\end{scope}
			\end{scope}	
		\end{tikzpicture}
		\caption{On the left: A part of the reference configuration $\R\times(-B,B)$ partitioned into the four colored subsets $E_1, E_2, E_3, E_4$ in \eqref{partition} for the maximal choice $\ffi = \arctan(\frac{1}{4})$; the dashed lines indicate the slip direction $s=e_2$. On the right: The image under the continuous piecewise affine map $w_\ffi$ as in \eqref{w_ffi_e_2}.}\label{fig:bend2}
	\end{figure}
\end{proof}

Now that we established how to ensure an adjustment in the angle while keeping the shear parameter for $s\in \{\pm e_1, \pm e_2\}$, it is time to address all remaining cases.
As we discussed earlier, this will be a direct consequence of Lemma \ref{lem:changes}.

\begin{lemma}[Changing angles for any slip direction]\label{lem:angle}
Let $B>0$ and $s\in\Scal^1$.
For every $|\theta|\leq \theta_{\max}$ (with $\theta_{\max}$ as in Lemma \ref{lem:angle_e1e2}) there exists a function $w\in W_\loc^{1,\infty}(\R\times (-B,B);\R^2)$, such that $\nabla w\in \Mcal_{s}$ a.e., and
\begin{align*}
	\nabla w(x_1,\cdot)=\begin{cases}
		\Id &\text{ for $x_1<-7 B,$}\\
		R_\theta &\text{ for $x_1>7 B$.}
	\end{cases}
\end{align*}
\end{lemma}
\begin{proof}
	The proof technically distinguishes between the two scenarios $s\cdot e_1 \geq \frac{1}{\sqrt{2}}$ and $s\cdot e_2 > \frac{1}{\sqrt{2}}$.
	
	First, we set $\tilde B=8B$ and apply Lemma \ref{lem:angle_e1e2} for $B=\tilde{B}$ to obtain $\tilde{w}\in W^{1,\infty}(\R\times (-8B,8B);\R^2)$ that satisfies $\tilde{w}_s\in\Mcal_{e_2}$ a.e.~and
	\begin{align*}
		\nabla w(x_1,\cdot)= \begin{cases}
			\Id &\text{ for }x_1<-\frac{1}{2}\tilde B = -4B,\\ 
			R_{\theta} &\text{ for } x_1>\frac{1}{2}\tilde B =4B.
		\end{cases}
	\end{align*}
	The final ingredient is Lemma \ref{lem:changes} applied to $w=\tilde{w}$ and $B=\tilde B$: If $s\cdot e_2 > \frac{1}{\sqrt{2}}$ we switch from $s=e_1$ to $\hat{s}=s$, otherwise we replace $e_2$ by $e_1$.
\end{proof}

Finally, we combine all previous building blocks to establish a Lemma about the construction of recovery sequences for Theorem \ref{theo:hard_constraint} on a thin strip with height $2h$.

\begin{lemma}[Change rotated shears]\label{lem:switch_rotated_shear}
	For any given $s\neq\pm e_1$, $\theta_1,\theta_2\in \T^1$, $\gamma_1,\gamma_2\in \R$, and $h>0$, there exist $r>0$ (independent of $h$) and a map $v\in W_\loc^{1,\infty}\big(\R\times (-h,h));\R^2\big)$ such that $\nabla v\in \Mcal_s$ and
	\begin{align*}
		\nabla v(y_1,\cdot)=\begin{cases}
			 \RS(\theta_1,\gamma_1) &\text{ for }y_1< -rh,\\
			 \RS(\theta_2,\gamma_2) &\text{ for }y_1>rh.
		\end{cases}
	\end{align*}
	Moreover, it holds that
	\begin{align}\label{Linfty_norm}
		\int_{-h}^h\int_{-rh}^{rh} |\nabla v|^2 \dd y \leq Ch^2
	\end{align}
	for a constant $C>0$ independent of $h$.
	
	If $s=\pm e_1$, then the same holds true if $\gamma_1=\gamma_2=0$.
\end{lemma}
\begin{proof}
This proof is essentially a combination of the previous building blocks that facilitate the switch from one rotated shear to another.
We begin with the easiest case $s=\pm e_1$, for which we assume that $\gamma_1=\gamma_2=0$, which means that we merely switch from one rotation $R_{\theta_1}$ to another $R_{\theta_2}$.

\smallskip

\textit{Step 1: The case $s=\pm e_1$.}
If not specified otherwise, we always apply the previous Lemmas for $A=B=h$ for $h>0$. 
Let $\theta_{\max}>0$ be as in Lemma \ref{lem:angle_e1e2}, set $n=\lceil\frac{|\theta_2-\theta_1|}{\theta_{\max}}\rceil$ and choose
\begin{align}\label{angles_iteration}
	0=\omega_1,\omega_2,\ldots,\omega_{n+1}=\theta_2-\theta_1\in \T^1 \quad\text{such that}\quad |\omega_{k+1} - \omega_k| \leq \theta_{\max}
\end{align}
In the following, we then construct inductively for $h>0$ and any $k\in\{1,\ldots,n\}$ a function $w_k\in W_\loc^{1,\infty}(\R\times(-h,h);\R^2)$ 
\begin{align}\label{w_k}
	\nabla w_k(y_1,\cdot)=\begin{cases}
		R_{-\omega_{k+1}} &\text{ for }y_1< -\tfrac{h}{2}+2h(k-1),\\
		\Id &\text{ for } y_1>\tfrac{h}{2} + 2h(k-1).
	\end{cases}
\end{align}
For $k=1$, we apply Lemma \ref{lem:angle_e1e2} for $\theta=\omega_2 - \omega_1 = \omega_2$ and subsequently Lemma \ref{lem:changetheta} for $\theta=-\omega_2$  (recall \eqref{angles_iteration}) to obtain $w_1\in W_\loc^{1,\infty}(\R\times(-h,h);\R^2)$ satisfying $\nabla w_1 \in \Mcal_{e_1}$, and
\begin{align*}
	\nabla w_1(y_1,\cdot)=\begin{cases}
		R_{-\omega_2} &\text{ for }y_1<-\tfrac{h}{2},\\
	    \Id &\text{ for } y_1>\tfrac{h}{2}.
	\end{cases}
\end{align*}
Now, assume that $w_k$ is already constructed for some $k\in\{1,\ldots,n-1\}$. We then design $w_{k+1}$ as follows:
An application of Lemma \ref{lem:angle_e1e2} for $\theta =\omega_{k+2}-\omega_{k+1}$, which can be done due to \eqref{angles_iteration}, and a translation of the argument by $2kh e_1$ yields $\tilde{w}_{k+1}\in W_\loc^{1,\infty}(\R\times (-h,h);\R^2\big)$ satisfying $\tilde{w}_{k+1}\in\Mcal_{e_1}$ and
\begin{align}\label{tildew_k+1}
	\nabla \tilde{w}_{k+1}(y_1,\cdot)=\begin{cases}
		\Id &\text{ for }y_1<-\tfrac{h}{2} + 2h k,\\
		R_{\omega_{k+2} - \omega_{k+1}} &\text{ for } y_1>\frac{h}{2}+2hk.
	\end{cases}
\end{align}
In light of \eqref{w_k} and \eqref{tildew_k+1}, we find that the map
\begin{align*}
	w_{k+1}(y_1,\cdot):= \begin{cases}
		R_{\omega_{k+1} - \omega_{k+2}}w_k(y_1,\cdot) &\text{ for } y_1 <-\frac{h}{2} + 2hk,\\ 
		R_{\omega_{k+1} - \omega_{k+2}}\tilde{w}_{k+1}(y_1,\cdot) + d_{k+1}&\text{ for } y_1\geq \frac{h}{2} + 2hk,
	\end{cases}
\end{align*}
is continuous for suitable $d_{k+1}\in \R^2$, contained in $W_\loc^{1,\infty}(\R\times(-h,h);\R^2)$, and satisfies $\nabla w_{k+1}\in\Mcal_{e_1}$ as well as the desired boundary conditions \eqref{w_k} for $k+1$.

Finally, we set 
\begin{align*}
	v(y):=R_{\theta_2}w_{n}(y + h (n-1)  e_1),\quad y\in \R\times (-h,h)
\end{align*}
to obtain the desired function with $r=n-\frac{1}{2}$.

\medskip

\textit{Step 2: The cases $s\neq \pm e_1$.} In these scenarios we have to additionally accommodate the shear parameter at the start and at the end.
We first proceed almost exactly as in Step 1, the only difference being that we apply Lemma \ref{lem:angle} instead of Lemma \ref{lem:angle_e1e2} for $B=h$.
This way, we produce some $\bar{r}>1$ and $\bar{v}\in W_\loc^{1,\infty}\big(\R\times (-h,h);\R^2\big)$ such that $\nabla v\in \Mcal_s$ and
\begin{align*}
	\nabla \bar{v}(y_1,\cdot)=\begin{cases}
		 \RS(\theta_1,0)=R_{\theta_1} &\text{ for }y_1<-\bar{r}h,\\
		 \RS(\theta_2,0)=R_{\theta_2} &\text{ for }y_1>\bar{r}h.
	\end{cases}
\end{align*}
Then, we apply Lemma \ref{lem:changegamma} for $A=rh$, $B=h$, $w=\bar{v}$, $\tilde{\gamma}_1=\gamma_1$, $\tilde{\gamma}_2=\gamma_2$ to generate a real number $r>\bar{r}>1$ depending on $s$, and a function $v\in W_\loc^{1,\infty}\big(\R\times (-h,h));\R^2\big)$ such that $\nabla v\in \Mcal_s$ and
\begin{align*}
	\nabla v(y_1,\cdot)=\begin{cases}
		 \RS(\theta_1,\gamma_1) &\text{ for }y_1<-rh,\\
		 \RS(\theta_2,\gamma_2) &\text{ for }y_1>rh.
	\end{cases}
\end{align*}

As we can see from the explicit constructions made in Lemmas \ref{lem:changes} - \ref{lem:angle}, the norm of $\nabla v$ in $L^\infty((-rh,rh)\times (-h,h);\R^{2\times 2})$ does not depend on $h$ and $v$ is piecewise affine; the same is true for the case $e=\pm e_1$. This proves the desired estimate \eqref{Linfty_norm}.
\end{proof}

\section{Proof of the main results}\label{sec:proofs}

\subsection{The $\Gamma$-limit with hard constraints}\label{sec:hard_constraint}

\begin{proof}[Proof of Theorem \ref{theo:hard_constraint}]
\textit{Step 1: Compactness.} Let $(u_h)_h\subset W^{1,2}(\Omega;\R^2)$ be a sequence with $\int_\Omega u_h \dd x = 0$ and $\sup_h \Ical_h(u_h) < \infty$.
In light of \eqref{constrained_density}, the latter yields that 
\begin{align}\label{bounded_energy}
	|\nabla^h u_h s| = 1 \text{ a.e.~in $\Omega$} \qand (\nabla^h u_h m)_h \text{ is bounded in $L^2(\Omega;\R^2)$.}
\end{align}
It is then evident that $(\nabla u_h)_h$ is bounded in $L^2(\Omega;\R^2)$, and thus, $(u_h)_h$ is bounded in $W^{1,2}(\Omega;\R^2)$ due to Poincar\'e's inequality in mean value form.
This is why we may select a subsequence (not relabeled) and $u\in W^{1,2}(\Omega;\R^2)$ such that $u_h\weakly u$ in $W^{1,2}(\Omega;\R^2)$. 
Moreover, we find that $\partial_2 u_h\to 0 = \partial_2 u$ in $L^2(\Omega;\R^2)$ due to the definition of the rescaled gradient; 
if $s=e_1$, then we even find that $|\partial_1 u| = |u'| \leq 1$ in view of \eqref{bounded_energy}, and the weak convergence $\partial_1 u_h \weakly \partial_1 u = u'$.

\medskip

\textit{Step 2: Lower bound.} Let $(u_h)_h$ and $u$ be as in Step 1. 
If $s=\pm e_1$, then the lower bound is trivial due to $|u'|\leq 1$ a.e.~in $(0,L)$ and
\begin{align*}
	\Ical_k(u_k) \geq 0 = \Ical(u).
\end{align*}
Otherwise, the energy satisfies
\begin{align*}
	\liminf_{h\to 0}\Ical_h(u_h) &= \liminf_{h\to 0}\int_\Omega W(\nabla^h u_h) \dd x \\
		&\geq \liminf_{h\to 0}\int_\Omega \Wmin^{\rm c}(\partial_1 u_h) \dd x \geq \int_\Omega \Wmin^{\rm c}(u') \dd x = 2 \int_0^L \Wmin^{\rm c}(u') \dd x.
\end{align*}
	
\textit{Step 3: Upper bound for $s\neq \pm e_1$.} 
The methodology for constructing recovery sequences is based on an approximation by piecewise affine functions together with Lemma \ref{lem:Wmin} and Lemma \ref{lem:switch_rotated_shear}. 
We first cover the cases $s\neq \pm e_1$, and deal with $s=\pm e_1$ later on.

\smallskip

\textit{Step 3a: Recovering piecewise affine limit deformations with large derivatives.}
Let $u:(0,L)\to \R^2$ be piecewise affine with $|u'|\geq 1$ almost everywhere. Then, there is a partition
\begin{align}\label{u_partition}
	0=t\ui{0}\leq t\ui{1}\leq \ldots \leq t\ui{N}=L
\end{align}
such that $u'$ is constant on each interval $(t\ui{n-1},t\ui{n})$ with $u\restrict{(t\ui{n-1},t\ui{n})}=:\xi\ui{n}$ for $n\in\{1,\ldots,N\}$.
In light of Lemma \ref{lem:Wmin} $b)$ and $|u'|\geq 1$ almost everywhere, we can find $\theta\ui{n}\in\T^1$ and $\gamma\ui{n}\in\R$ such that
\begin{align}\label{representation}
	u'=\xi\ui{n} = \RS(\theta\ui{n},\gamma\ui{n})e_1\qand \Wmin^c(u') = (\gamma\ui{n})^2 \text{ on } (t\ui{n-1},t\ui{n}).
\end{align}
The goal is now to fatten $u$ suitably. In view of Lemma \ref{lem:switch_rotated_shear}, we find $r>0$ and a deformation $v_h\ui{n}\in W^{1,\infty}\big((t\ui{n} -2rh, t\ui{n}+2rh)\times (-h,h);\R^2\big);\R^2)$ with $n\in \{1,\ldots, N-1\}$ such that $\nabla v_h\in \Mcal_s$ a.e.~in $(t\ui{n} -2rh, t\ui{n}+2rh)\times (-h,h)$ and
\begin{align*}
	\nabla v_h^{(n)}(y_1 + t\ui{n},\cdot) = \begin{cases}
		\RS(\theta\ui{n},\gamma\ui{n}) & \text{ if } y_1 <-rh,\\
		\RS(\theta\ui{n+1},\gamma\ui{n+1}) & \text{ if } y_1 >rh.
	\end{cases}
\end{align*}
For $h$ sufficiently small, we set $\Gamma_h := \bigcup_{n=1}^{N-1} (t\ui{n}-rh,t\ui{n}+rh)\subset (0,L)$ and define $v_h\in W^{1,2}(\Omega_h;\R^2)$ as a continuous and piecewise affine function with gradients
\begin{align*}
	\nabla v_h(y_1,\cdot) := \begin{cases}
					\RS(\theta\ui{n},\gamma\ui{n}) &\text { if } y_1 \in [t\ui{n-1},t\ui{n})\setminus\Gamma_h \text{ for some } n\in \{1,\ldots,N\},\\
					\nabla v_h^{(n)}(y_1,y_2) &\text{ if } y_1 \in [t\ui{n-1},t\ui{n})\cap\Gamma_h \text{ for some } n\in \{1,\ldots,N\}.
					\end{cases}
\end{align*}
Note that this construction yields $\nabla v_h\in \Mcal_s$ a.e.~in $\Omega_h$.

In light of \eqref{u_partition} and the estimate \eqref{Linfty_norm} for $v_h^{(n)}$, we find for the energy of $\Ecal_h(v_h)$ as in \eqref{thin_energy} that
\begin{align*}
	\Ecal_h(v_h) &= \int_{\Omega_h} W(\nabla v_h) \dd y =  \int_{-h}^h \int_{(0,L)\setminus \Gamma_h}W(\nabla v_h) \dd y + \int_{-h}^h \int_{\Gamma_h}W(\nabla v_h) \dd y\\
	&=\int_{-h}^h \int_{(0,L)\setminus \Gamma_h}\Wmin^{\rm c}(u') \dd y + \int_{-h}^h \int_{\Gamma_h}|\nabla v_h m|^2  -1 \dd y\\
	&\leq 2h\int_0^L \Wmin^{\rm c}(u') \dd x_1 + C(N-1)h^2
\end{align*}
for a constant $C>0$ independent of $h$.
Passing to the limit as $h\to 0$ then yields that
\begin{align*}
	\limsup_{h\to 0} \frac{1}{h}\Ecal_h(v_h) \leq  2\int_0^L \Wmin^{\rm c}(u') \dd x_1 = \Ical(u).
\end{align*}
Finally, we invoke the change of variables \eqref{rescaling} to generate a sequence $(u_h)_{h}\subset W^{1,\infty}(\Omega;\R^2)$ that converges to $u$ uniformly on $\Omega$ and has the same (rescaled) energy as $(v_h)_{h}$, that is, $\Ical_h(u_h) = \frac{1}{h}\Ecal_h(v_h)$.
In other words, $(u_h)_h$ is a recovery sequence for the piecewise affine function $u$.

\medskip

\textit{Step 3b: Approximation by piecewise affine functions and relaxation.}
Let $u\in  \Acal$ and find a sequence $(\tilde v_j)_j$ of piecewise affine functions such that $\tilde v_j \to u$ in $\Acal$. 
In light of the quadratic growth of $\Wmin^c$, see \eqref{Wmin_growth} in Lemma \ref{lem:Wmin}, it holds that
\begin{align*}
	\lim_{j\to \infty} \int_0^L \Wmin^c(\tilde v_j') \dd x_1 = \int_0^L \Wmin^c(u')\dd x_1.
\end{align*}
For every $j\in\N$, there is a partition $0=t_j\ui{0}\leq t_j\ui{1}\leq \ldots \leq t\ui{N_j}_j=L$ such that $\tilde v_j'= \xi_j\ui{n}$ on $(t\ui{n-1}_j,t\ui{n}_j)$ for all $n\in \{1,\ldots,N_j\}$. 
As we have seen in Step 3a, the cases $|\xi_j\ui{n}|\geq 1$ produce a non-trivial quadratic energy contribution; we now deal with the cases $|\xi_j\ui{n}|\leq 1$, which can be obtained as weak limits of suitable piecewise affine functions.

Let $J_j=\{n\in\{1,\ldots, N_j\} : |\xi_j\ui{n}|<1\}$ and let $n\in J_j$ be fixed but arbitrary.
Then, there exist $R_j\ui{n}\in\SO(2)$ and $\theta_j\ui{n}\in\T^1$ such that $\xi_j\ui{n} = \cos\theta_j\ui{n}R_j\ui{n}e_1$.
Now, define for $i\in\N$,
\begin{align*}
	\ffi_{j,i}\ui{n} : I_j\ui{n}:=(t_j\ui{n-1},t_j\ui{n}) \to \R^2,\: t\mapsto \tfrac{1}{i}\ffi_j\ui{n}(it)
\end{align*}
where $\ffi_j\ui{n} : I_j\ui{n} \to \R^2$ is piecewise affine with $\ffi_j\ui{n}= \tilde v_j$ on $\partial I_j\ui{n}$, and gradients
\begin{align*}
	(\ffi_j\ui{n})'=\begin{cases}
					R_j\ui{n}R_{\theta_j\ui{n}}e_1 \text{ on } (t_j\ui{n-1}, \tau_j\ui{n}),\\
					R_j\ui{n}R_{-\theta_j\ui{n}}e_1 \text{ on } [\tau_j\ui{n},t_j\ui{n}),
				\end{cases}
\end{align*}
with $\tau_j\ui{n}= \frac{1}{2}(t_j\ui{n-1}+t_j\ui{n})$.
It then holds that $\ffi_{j,i}\ui{n}\weakly \tilde v_j$ in $W^{1,2}(I_j\ui{n};\R^2)$ for $i\to \infty$. Indeed, the Riemann-Lebesgue Lemma yields that
\begin{align*}
	(\ffi_{j,i}\ui{n})'\weakly \frac{1}{|I_j\ui{n}|}\int_{I_j\ui{n}} (\ffi_j\ui{n})'\dd t = \frac{1}{2}R_j\ui{n}(R_{\theta_j\ui{n}} + R_{-\theta_j\ui{n}})e_1 = R_j\ui{n}\cos \theta_j\ui{n} e_1 = \xi_j\ui{n}
\end{align*}
in $L^2(I_j\ui{n};\R^2)$, and due to $\ffi_{j,i}\ui{n} = \tilde v_j\restrict{I_j\ui{n}}$ on $\partial I_j\ui{n}$ we may apply Poincar\'e's inequality to conclude the desired convergence.
We are thus lead to consider the piecewise affine functions $v_{j,i} : (0,L)\to \R^2$ given by
\begin{align*}
	v_{j,i} = \tilde v_j + \sum_{n\in J_j}\mathbbm{1}_{I_j\ui{n}}(\ffi_{j,i}\ui{n} - \tilde v_j),
\end{align*}
which satisfy
\begin{align*}
	v_{j,i}\weakly v_j \text{ in } W^{1,2}(\Omega;\R^2)\text{ for $i\to \infty$}\qand |v_{j,i}'|\geq 1 \text{ a.e.~in } (0,L) \text{ for all $j,i\in\N$,}
\end{align*}
as well as
\begin{align*}
	\lim_{j\to\infty}\lim_{i\to \infty}\int_0^L \Wmin^c(v_{j,i}') \dd x_1 = \int_0^L \Wmin^c(u') \dd x_1.
\end{align*}
Since $\Wmin^c$ is coercive \eqref{Wmin_growth}, we may select a diagonal sequence $v_j = v_{j,i(j)}$ such that
\begin{align}\label{piecewiese_diagonal}
	v_{j}\weakly u \text{ in } W^{1,2}(\Omega;\R^2),\quad \lim_{j\to\infty}\Ical(v_j) = \Ical(u), \quad |v_{j}'|\geq 1.
\end{align}

\smallskip

\textit{Step 3c: Diagonal sequence.}
Let $u\in \Acal$ and $v_j$ be as in Step 3b. By applying Step 3a for $v_j$ with $j\in\N$, we obtain a sequence $(u_{j,h})_h\subset W^{1,2}(\Omega;\R^2)$ that satisfies, under consideration of \eqref{piecewiese_diagonal},
\begin{align*}
	\lim_{j\to\infty}\lim_{h\to 0}\norm{u_{j,h} -u}_{L^2(\Omega;\R^2)} = 0\qand \limsup_{j\to\infty}\limsup_{h\to 0} \Ical_h(u_{j,h}) \leq \Ical(u). 
\end{align*}
This allows us to select a diagonal sequence in the sense of Attouch $(u_h)_h = (u_{j(h),h})_h$ such that
\begin{align*}
	u_h\to u \text{ in } L^2(\Omega;\R^2)\qand \limsup_{k\to\infty}\Ical_h(u_h) \leq \Ical(u).
\end{align*}
Finally, we repeat Step 1 for the sequence $(u_h)_h$ to conclude that $u_h\weakly u$ in $W^{1,2}(\Omega;\R^2)$ due to the uniqueness of weak limits.	
	
\medskip

\textit{Step 4: Upper bound for $s=\pm e_1$.} 
We deal with this case by selecting, with the help of standard results in the context of asymptotic and convex analysis, for $u\in \Acal$ with $|u'|\leq 1$ a.e.~in $(0,L)$ a sequence $(\tilde v_j)_j$ of piecewise affine functions with $|\tilde v_j|=1$ a.e.~in $(0,L)$ such that $\tilde v_j \weakly u$ in $W^{1,2}(\Omega;\R^2)$ similar to Step 3b.
Then, apply Step 3a for $u=v_j$ for each $j\in \N$ to obtain $u_{j,h}\in W^{1,2}(\Omega;\R^2)$ such that 
\begin{align*}
	\lim_{j\to\infty}\lim_{h\to 0}\norm{u_{j,k}-u}_{L^2(\Omega;\R^2)}\qand \limsup_{h\to 0} \Ical_h(u_{j,h}) = 0 = \Ical(u) \text{ for every } j\in \N.
\end{align*}
The rest can be handled exactly as in Step 3c.
\end{proof}

\subsection{The $\Gamma$-limit with soft constraints}\label{sec:soft_constraint}			

\begin{proof}[Proof of Corollary \ref{cor:soft_constraint}]
The methodology of the lower bound is primarily inspired by \cite[Section 4.1]{CDK13a}. 
For the reader's convenience, we provide a self contained proof.

\textit{Step 1: Compactness.}
Exactly as in \cite[Equation (3.1)]{CDK11}, we rewrite every $Q\in\SO(2)$ as $Q=a\otimes s + a^\perp \otimes m$ with $a\in \Scal^1$, and the energy density $W_{\eps_k}$ as
\begin{align}\label{decomposition}
	W_{\eps_k}(F) &= \min_{\gamma\in \R,\, a\in\Scal^1}\big(\frac{1}{\eps_k}|F(s\otimes s + m\otimes m - \gamma s\otimes m) - a\otimes s - a^\perp\otimes m|^2 + \gamma^2\big)\nonumber\\
	&=\frac{1}{{\eps_k}}\big(|Fs - a_{\eps_k}(F)|^2 + |Fm - \gamma_{\eps_k}(F)Fs - a_{\eps_k}^\perp(F)|^2\big) + \gamma_{\eps_k}^2(F),
\end{align}
where $a_{\eps_k}(F)\in\Scal^1$ and $\gamma_{\eps_k}(F)\in\R$ are the minimizers that define $W_{\eps_k}(F)$.

Let $(u_k)_k\subset W^{1,1}(\Omega;\R^2)$ with bounded energy and vanishing mean value.
In the following, we write briefly $a_k:=a_{\eps_k}(\nabla^{h_k}u_k)$ and $\gamma_k = \gamma_{\eps_k}(\nabla^{h_k}u_k)$.
The rescaled gradient of $u_k$ can be written as
\begin{align*}
	\nabla^{h_k} u_k = A_k + B_k
\end{align*}
where
\begin{align}\label{def_A_k}
	\begin{split} 
		A_k &= a_k\otimes s + (\gamma_ka_k+a_k^\perp)\otimes m\\
		B_k &= (\partial_s^{h_k} u_k - a_k)\otimes s + (\partial_m^{h_k} u_k -\gamma_{\eps_k}\partial_s^{h_k}u_k - a_k^\perp)\otimes m + \gamma_k(\partial_s^{h_k}u_k - a_k)\otimes m,
	\end{split}
\end{align}
with $\partial_s^{h_k} u_k = (\nabla^{h_k} u_k)s$ and $\partial_m^{h_k} u_k = (\nabla^{h_k} u_k)m$.
In light of \eqref{decomposition}, we deduce that $(\gamma_k)_k\subset L^2(\Omega)$ is bounded and that
\begin{align*}
	\partial_s^{h_k} u_k - a_k \to 0,\quad \partial_m^{h_k} u_k -\gamma_{\eps_k}\partial_s^{h_k}u_k - a_k^\perp \to 0 \text{ in }L^2(\Omega;\R^2).
\end{align*}
Via H\"older's inequality, we conclude that $\gamma_k(\partial_s^{h_k}u_k - a_k)\to 0$ in $L^1(\Omega;\R^2)$. Finally, since $(a_k)_k\subset L^\infty(\Omega;\R^2)$, we may select a subsequence (not relabeled), such that
\begin{align}\label{conv_A_k}
	A_k\weakly A \text{ in } L^2(\Omega;\R^{2\times 2}),\quad B_k\to 0 \text{ in } L^1(\Omega;\R^{2\times 2}),
\end{align}
for some $A\in L^2(\Omega;\R^{2\times 2})$, which implies that $\nabla^{h_k} u_k \weakly A$ in $L^1(\Omega;\R^2)$.
In particular, we obtain that $\partial_2 u_k \to 0$ in $L^1(\Omega;\R^2)$.
Since $u_k$ has vanishing mean value, we infer from Poincar\'e's inequality that $u_k\weakly u$ in $W^{1,1}(\Omega;\R^2)$ for $u\in \Acal$.

Moreover, if $s=\pm e_1$, then $|u'|\leq 1$ almost everywhere as the weak limit of $\partial_1 u_k = A_ke_1$, which satisfies $|A_ke_1|=1$ almost everywhere.

\medskip

\textit{Step 2: Lower bound.} Let $(u_k)_k \subset W^{1,1}(\Omega;\R^2)$ and $u$ as in Step 1.
The lower bound for $s=\pm e_1$ is exactly as in the proof of Theorem \ref{theo:hard_constraint}.
As for $s\neq \pm e_1$, we combine \eqref{constrained_density} with \eqref{decomposition}, \eqref{def_A_k} and \eqref{conv_A_k}, to conclude that
\begin{align*}
	\liminf_{k\to\infty}\Ical_k(u_k) &= \liminf_{k\to\infty}\int_\Omega W_{\eps_k}(\nabla^{h_k} u_k) \dd x \geq \liminf_{k\to\infty}\int_\Omega \gamma_k^2 \dd x = \liminf_{k\to\infty}\int_\Omega W(A_k) \dd x \\
	&\geq \liminf_{k\to\infty}\int_\Omega \Wmin^c(A_ke_1) \dd x \geq \int_\Omega \Wmin^c(Ae_1) \dd x = 2\int_0^L \Wmin^c(u') \dd x_1 =\Ical(u). 
\end{align*}

\smallskip

\textit{Step 3: Upper bound.} We proceed exactly as in the proof of Theorem \ref{theo:hard_constraint}. 
The only difference is that the different growth of $\Ical_k$ yields, in Steps 3c and 4, merely the weak convergence of the diagonal sequence $u_k\weakly u$ in $W^{1,1}(\Omega;\R^2)$ as opposed to $W^{1,2}(\Omega;\R^2)$.
\end{proof}

The following proposition shows that, under suitable scaling relations between the parameters $\eps$ and $k$, which govern the diverging elastic energy contribution and the thickness of the film, the recovery-sequence for Corollary \ref{cor:soft_constraint} can be made smooth. 
However, we will see in Section \ref{sec:gap} that this is not always possible for other scaling regimes and non-smooth sequences are required to produce optimal energy.

\begin{proposition}\label{prop:ratio}
	Let $(\eps_k,h_k)_k$ be any sequence with $(\eps_k,h_k)\to (0,0)$ such that
	\begin{align}\label{ratio}
		\frac{h^{\alpha}_k}{\eps_k}\to 0
	\end{align}
	for some $\alpha\in (0,2)$.
	Then, the recovery sequence for the $\Gamma$-convergence result in Corollary \ref{cor:soft_constraint} can be chosen to be smooth.
\end{proposition}
\begin{proof}
The methodology for constructing recovery sequences is based on the techniques used in the context of 3d-1d dimension reduction results for hyperelastic strings \cite[Theorem 4.5]{Sca06}, see also \cite[Theorem 4.1]{EnK21a}. 
The case $s=\pm e_1$ can be dealt with in the same way as in the aforementioned works (with only minor adjustments) due to the trivial structure of the $\Gamma$-limit and $W_{\eps_k}(F)\leq \frac{1}{\eps_k}\dist^2(F,\SO(2))$ for all $F\in\R^{2\times 2}$. We shall thus assume that $s\neq \pm e_1$.

As before, we begin by recovering piecewise affine limit deformations.
Let $u:(0,L)\to \R^2$ be piecewise affine with $|u'|\geq 1$ a.e.~in $(0,L)$ and recall \eqref{u_partition} and \eqref{representation} for suitable $\theta\ui{n}\in \T^1$ and $\gamma\ui{n}\in\R$.

Set $\beta = 2-\alpha\in (0,2)$, where $\alpha$ is given as in \eqref{ratio} and for $k$ large enough, we define the disjoint union $\Gamma_k = \bigcup_{n=1}^{N-1}[t\ui{n},t\ui{n}+h_k^\beta)$.
We then set the smooth functions $\theta_k : \Gamma_k\to\SO(2)$ and $\gamma_k: \Gamma_k \to \R$ such that $\theta_k = \theta\ui{n}$ and $\gamma_k = \gamma\ui{n}$ in a neighborhood around $t\ui{n}$, $\theta_k = R\ui{n+1}$ and $\gamma_k = \gamma\ui{n+1}$ in a neighborhood around $t\ui{n} + h_k^\beta$, as well as $\theta_k' = 0$ and $\gamma_k'=0$ when near the boundary of $\Gamma_k$.
Then, these sequences satisfy
\begin{align*}
	\norm{\theta_k}_{L^{\infty}(\Gamma_k;\R^{2\times 2})} = \Ocal(1),\quad \norm{\theta_k'}_{L^{\infty}(\Gamma_k;\R^{2\times 2})} = \Ocal(h_k^{-\beta}) 
\end{align*}
as well as
\begin{align*}
	\norm{\gamma_k}_{L^{\infty}(\Gamma_k)} = \Ocal(1),\quad \norm{\gamma_k'}_{L^{\infty}(\Gamma_k)} = \Ocal(h_k^{-\beta}) .
\end{align*}

A recovery sequence for $u$ is then given by
\begin{align*}
	u_k(x) = \begin{cases}
					\displaystyle \RS(\theta\ui{1},\gamma\ui{1})x_{h_k} + b_k\ui{1}	&\text{if } x_1\in [0,t\ui{1})\\[0.2cm]
					\displaystyle \int_{t\ui{n}}^{x_1} \RS(\theta_k(t),\gamma_k(t))e_1 \dd t  + d_k\ui{n} &\text{if } x_1\in [t\ui{n},t\ui{n}+h_k^\beta) \text{ for } n=1,\ldots N-1\\
					\displaystyle \quad + h_k x_2\RS(\theta_k(x_1),\gamma_k(x_1))e_2\\[0.2cm]
					\displaystyle \RS(\theta\ui{n+1},\gamma\ui{n+1})x_{h_k}x_{h_k} +b_k\ui{n}	&\text{if }  x_1\in [t\ui{n}+h_k^\beta,t\ui{n+1}) \text{ for } n=1,\ldots,N-1
			\end{cases}
\end{align*}
for $x\in\Omega$ where $x_{h_k} = x_1e_1 + h_k x_2e_2$; the translations $b_k\ui{n},d_k\ui{n}\in\R^2$ make $u_k$ continuous.
The rescaled gradients of $u_k$ then have the form
\begin{align*}
	\nabla^{h_k} u_k(x) = \begin{cases}
					\displaystyle \RS(\theta\ui{1},\gamma\ui{1}) &\text{if } x_1\in [0,t\ui{1})\\[0.2cm]
					\displaystyle \RS(\theta_k(x_1),\gamma_k(x_1))+h_k E(x) &\text{if }x_1\in [t\ui{n},t\ui{n}+h_k^\beta)\text{ for } n=1,\ldots N-1\\[0.2cm]
					\displaystyle \RS(\theta\ui{n+1},\gamma\ui{n+1})	&\text{if }  x_1\in [t\ui{n}+h_k^\beta,t\ui{n+1}) \text{ for } n=1,\ldots,N-1
			\end{cases}
\end{align*}
with
\begin{align}\label{auxE}
	E(x) = x_2\big(\frac{d}{d x_1} R_{\theta_k(x_1)}+ \frac{d}{d x_1}(\gamma_k(x_1)R_{\theta_k(x_1)}) s\otimes m\big)e_2\otimes e_2
\end{align}
for $x\in\Omega$.
By design, we obtain that
\begin{align*}
	u_k \to u \text{ in } W^{1,2}(\Omega;\R^2).
\end{align*}
Joining \eqref{approx_densities} with \eqref{representation} - \eqref{auxE}, and $|\Gamma_k| = h_k^\beta$, we compute that
\begin{align*}
	\Ical_k(u_k) &=  \int_\Omega W_{\eps_k}(\nabla^{h_k} u_k) \dd x 
			\leq 2 \int_{(0,L)\setminus \Gamma_k} \Wmin^c(u') \dd x_1 \\ 
		&\quad+ \int_{\Gamma_{k}\times (-1,1)} \frac{1}{\eps_k}\dist^2(R_{\theta_k(x_1)} + h_k E(x)(\Id - \gamma_k(x_1) s\otimes m),\SO(2)) + \gamma_k(x_1)^2 \dd x\\
		&\leq 2\int_0^L \Wmin^c(u') \dd x_1 + \int_{\Gamma_{k}\times (-1,1)} \frac{1}{\eps_k}h_k^2|E(x)(\Id - \gamma_k(x_1) s\otimes m)|^2 + \gamma_k(x_1)^2 \dd x\\
		&\leq \Ical(u) + \Ocal\Big(\frac{h_k^{2-\beta}}{\eps_k}\Big) + \Ocal\big(h_k^\beta).
\end{align*}
Finally, we exploit \eqref{ratio} to conclude that
\begin{align*}
	\limsup_{k\to\infty}\Ical_k(u_k) \leq \Ical(u).
\end{align*}
The rest can be done exactly as in Steps 3b-4 in the proof of Theorem \ref{theo:hard_constraint}.
\end{proof}

\subsection{The smoothness gap of the hard contraint}\label{sec:gap}

In this section, we will see that the hard constraint case for $s=e_1$ exhibits some unexpected extra rigidity
when constrained to smoother functions preventing jumps of the rotation or shear parameter, leading to a Lavrentiev phenomenon. The basis of this is the following necessary condition 
obtained by combining the constraint with gradient structure.
\begin{lemma}\label{lem:HCcurl}
Suppose that $\Omega$ is a bounded domain in $\R^2$ and $v\in W^{1,1}(\Omega;\R^2)$
satisfies the hard constraint for $s=\pm e_1$, i.e.,
\begin{align}\label{HC}  
\begin{aligned}
	\nabla v(x)=\RS(\theta(x),\gamma(x);\pm e_1)=R_{\theta(x)}(\Id + \gamma(x) e_1\otimes e_2)\quad\text{for a.e.}~x\in \Omega.
\end{aligned}
\end{align}
with $\gamma\in W^{1,1}(\Omega;\R)\cap C^0(\Omega;\R)$ and $\theta\in W^{1,1}(\Omega;\T^1)\cap C^0(\Omega;\T^1)$. Then 
\begin{align}\label{HCcurl}  
\begin{aligned}
	\partial_1 \theta = \partial_1 \gamma
	\quad\text{and}\quad
	\partial_2 \theta = \gamma \partial_1 \theta\quad\text{in $\Omega$},
\end{aligned}
\end{align}
in the sense of distributions.
\end{lemma}
\begin{proof}
Clearly, $\Curl \nabla v=0$ in $\Omega$ (row-wise $\curl$) in the sense of distributions. Expressing this in terms of $\theta$ and $\gamma$, we infer from \eqref{HC} that
\begin{align}\label{lem:HCc-1}  
	0=\partial_2 \vecr{\cos \theta\\\sin\theta}-\partial_1\vecr{\gamma\cos \theta\\\gamma\sin\theta}-\partial_1\vecr{-\sin\theta\\\cos\theta}.
\end{align}
Exploiting that all appearing quantities are contained in $W^{1,1}(\Omega;\R)\cap C^0(\Omega;\R)$, we observe by density of smooth functions that
\begin{align}\label{lem:HCc-2}  
\begin{aligned}
	\partial_2 \vecr{\cos \theta\\\sin \theta}=(\partial_2 \theta)\vecr{-\sin\theta\\\cos\theta},~~~
	\partial_1\vecr{-\sin\theta\\\cos\theta}=-(\partial_1 \theta)\vecr{\cos\theta\\\sin\theta},~~~\text{and}&\\
	\partial_1\vecr{\gamma\cos\theta\\\gamma\sin\theta}=
	(\partial_1 \gamma)\vecr{\cos \theta\\\sin \theta}+(\partial_1 \theta)\gamma\vecr{-\sin\theta\\\cos\theta}&.
\end{aligned}
\end{align}
Since $(\cos\theta,\sin\theta)$ and $(-\sin\theta,\cos\theta)$ are orthonormal, \eqref{lem:HCc-1} and \eqref{lem:HCc-2} combined yield \eqref{HCcurl}.
\end{proof}

The first equation in \eqref{HCcurl} couples shear and bending of the fibers (the lines in direction $e_1$ in $\Omega$). This implies a minimal energy cost associated to bending,
for instance if it is forced by shortening the distance between the ``ends'' of $v$ (at $x_1=0$ and $x_1=L$).

\begin{proposition}\label{prop:smoothturn}
Recall $\Omega_h:=(0,L)\times (-h,h)$, $L,h>0$, and let $v\in W^{2,1}(\Omega_h;\R^2)\cap C^{1}(\overline\Omega_h;\R^2)$ satisfy \eqref{HC} on $\Omega_h$.
In particular, $\gamma\in W^{1,1}(\Omega_h;\R)\cap C^{0}(\overline\Omega_h;\R)$ and 
$\theta\in W^{1,1}(\Omega_h;\T^1)\cap C^0(\overline{\Omega}_h;\T^1)$.
If $v$ satisfies
\begin{align}\label{squeeze}
	\frac{1}{L}\abs{v(0,\cdot)-v(L,\cdot)}\leq \delta
\end{align}
for some $\delta\in [0,1)$, then $\Ecal_{h}(v)=\int_{\Omega_h} \gamma^2\dd y\geq c h$ for a constant $c>0$ independent of $h$ and $v$, cf.~\eqref{density_slip}, \eqref{thin_energy}, and \eqref{constrained_density}.
\end{proposition}
\begin{proof}
For $\eta\in(0,1)$ (to be chosen later), define
\[ 
	G_\eta=G_\eta(y_2,\gamma):= \big\{y_1\in (0,L) \mid \eta\leq |\gamma(y_1,y_2)|\big\},~~~y_2\in (-h,h).
\]
This allows us to estimate the energy as follows:
\begin{align*}
	\Ecal_{h}(v)=\int_{-h}^h\int_0^L \gamma^2\dd y_1 \dd y_2	\geq \int_{-h}^h\int_{G_\eta(y_2)} \eta^2 \dd y_1 \dd y_2 = \eta^2 \int_{-h}^h |G_\eta|\dd y_2.
\end{align*}
If $|G_\eta|\geq \eta$ for a.e.~$y_2\in(-h,h)$, we obtain the assertion for $c:=2\eta^3$.

Otherwise, there exists a set of positive measure $V\subset (-h,h)$ such that $|G_\eta|<\eta$ for all $y_2\in V$. 
We claim that for a suitably small choice of $\eta=\eta(L,\delta)$, this case is impossible, essentially because it is not compatible with the ``short'' boundary conditions on $v$ \eqref{squeeze} combined with Lemma~\ref{lem:HCcurl}.

Due to the first equation in \eqref{HCcurl}, there exists $\alpha:(-h,h)\to\R$ measurable such that for all $y=(y_1,y_2)\in\Omega_h$, 
\begin{align}\label{gamma_theta}
	\gamma(y)=\theta(y)+\alpha(y_2)
	\quad\text{for a.e.~$y=(y_1,y_2)\in\Omega_h$.}
\end{align}
For each $y_2\in (-h,h)$, choose an orthonormal basis $b_1(y_2),b_2(y_2)$ of $\R^2$ such that $\det(b_1|b_2)=1$ and $b_1$ points in direction of $v(0,y_2)-v(L,y_2)$, i.e., 
\begin{align}\label{p-smoothturn-1}
	b_1(y_2) \cdot(v(0,y_2)-v(L,y_2))=\abs{v(0,y_2)-v(L,y_2)},\quad
	b_2(y_2) \cdot(v(0,y_2)-v(L,y_2))=0
\end{align}                                                      
and
\begin{align*}
	(b_1|b_2)\in \SO(2),~~~\text{whence $(b_1|b_2)=R_{\theta_0}$ for some $\theta_0=\theta_0(y_2)$.}
\end{align*}                                                       
We set
\[ 
	\hat\theta(y):=\theta(y)+\theta_0(y_2)\qand \hat\alpha(y_2):=\alpha(y_2)-\theta_0(y_2),\quad\text{ for }y\in\Omega_h,
\]
and find with basic computations and \eqref{gamma_theta} that
\[
	\vecr{\cos \hat{\theta}(y)\\ \sin \hat{\theta}(y)}=R_{\hat\theta(y)}e_1 = R_{\theta_0(y_2)}R_{\theta(y)} e_1\qand \gamma(y)=\hat\theta(y)+\hat\alpha(y_2).
\]                                     
Moreover, for all $y_2\in (-h,h)$, 
\begin{align}\label{p-smoothturn-2}
	\frac{1}{L}\int_0^L \sin\hat{\theta}\dd y_1=\frac{1}{L}\int_0^L b_2\cdot \vecr{\cos\theta\\\sin\theta}\dd y_1 =\frac{1}{L}b_2\cdot\int_0^L \partial_1 v\dd y_1=0,
\end{align}
the latter by \eqref{p-smoothturn-1}.
Similarly, \eqref{p-smoothturn-1} and \eqref{squeeze} give that
\begin{align}\label{p-smoothturn-3}
	0\leq \frac{1}{L}\int_0^L \cos\hat{\theta} \dd y_1=\frac{1}{L} b_1\cdot(v(0,\cdot)-v(L,\cdot))\leq \delta.
\end{align}
On the other hand, 
\[
	\gamma(y)=\hat{\theta}(y)+\hat{\alpha}(y_2)\in (-\eta,\eta)
	~~~\text{for}~~y_1\in (0,L)\setminus G_\eta(y_2), 
\]
by definition of $G_\eta$. If $y_2\in V$ so that $|G_\eta(y_2)|<\eta$, we obtain that
\begin{align}\label{p-smoothturn-4}
\begin{aligned}
	\int_0^L |\sin\hat{\theta}-\sin(-\hat\alpha)| \dd y_1 &\leq 2|G_\eta|+\int_{(0,L)\setminus G_\eta} |\sin(\gamma-\hat{\alpha})-\sin(-\hat{\alpha})|\dd y_1 \\
	&< 2\eta + \int_{(0,L)\setminus G_\eta} |\gamma|\dd y_1< 2\eta(1+L).
\end{aligned}
\end{align}
Analogously, we can estimate
\begin{align}\label{p-smoothturn-5}
	\int_0^L |\cos\hat{\theta}-\cos(-\hat\alpha)| \dd y_1 < 2\eta(1+L).
\end{align}
for all $y_2\in V$. As $\hat\alpha=\hat\alpha(y_2)$ is independent of $y_1$, we can combine \eqref{p-smoothturn-2} and \eqref{p-smoothturn-3} with \eqref{p-smoothturn-4} and \eqref{p-smoothturn-5}, respectively, to get that for all $y_2\in V$,
\begin{align}\label{p-smoothturn-6}
  |\sin(-\hat\alpha)|< 2\eta(1+L)\qand -2\eta(1+L)\leq\cos(-\hat\alpha)\leq \delta+2\eta(1+L).
\end{align}
Here, notice that since $\delta<1$, \eqref{p-smoothturn-6} for small $\eta$ means that $\hat\alpha$ has to stay close to the set $\pi\Z=\sin^{-1}(\{0\})$ 
while keeping a distance from $\pi\Z=\cos^{-1}(\{\pm 1\})$, the same set.
Of course, this is impossible if $\eta=\eta(\delta,L)$ is small enough.
\end{proof}

\subsection*{Acknowledgments}
This work had started while DE was still affiliated with Universiteit Utrecht. 
DE acknowledges the financial support by the NDNS+ Ph.D.~travel grant and would like to thank Carolin Kreisbeck for suggesting this project and sharing some initial ideas. The work of SK and MK were supported by the GA \v{C}R-FWF project 21-06569K.

%%%%%%%%%%%%%%%%%%%%%%%% BIBLIOGRAPHY %%%%%%%%%%%%%%%%%%%%%%%%%%%

\bibliographystyle{abbrv}
\bibliography{EnglKroemerKruzik2023_Bib}

\begin{thebibliography}{10}

\bibitem{CHM02}
C.~Carstensen, K.~Hackl, and A.~Mielke.
\newblock Non-convex potentials and microstructures in finite-strain
  plasticity.
\newblock {\em R. Soc. Lond. Proc. Ser. A Math. Phys. Eng. Sci.},
  458(2018):299--317, 2002.

\bibitem{ChK17}
F.~Christowiak and C.~Kreisbeck.
\newblock Homogenization of layered materials with rigid components in
  single-slip finite crystal plasticity.
\newblock {\em Calc. Var. Partial Differential Equations}, 56(3):Paper No. 75,
  28, 2017.

\bibitem{Con06}
S.~Conti.
\newblock Relaxation of single-slip single-crystal plasticity with linear
  hardening.
\newblock In {\em Proceedings of Multiscale Material Modeling Conference.
  Freiburg}, pages 18--22, 2006.

\bibitem{CDOR18}
S.~Conti, L.~F. Djodom, M.~Ortiz, and C.~Reina.
\newblock Kinematics of elasto-plasticity: validity and limits of applicability
  of {$\bold{F}=\bold{F}^{\rm e}\bold{F}^{\rm p}$} for general
  three-dimensional deformations.
\newblock {\em J. Mech. Phys. Solids}, 121:99--113, 2018.

\bibitem{CoD20}
S.~Conti and G.~Dolzmann.
\newblock Quasiconvex envelope for a model of finite elastoplasticity with one
  active slip system and linear hardening.
\newblock {\em Contin. Mech. Thermodyn.}, 32(4):1187--1196, 2020.

\bibitem{CoD21}
S.~Conti and G.~Dolzmann.
\newblock Optimal laminates in single-slip elastoplasticity.
\newblock {\em Discrete Contin. Dyn. Syst. Ser. S}, 14(1):1--16, 2021.

\bibitem{CDK11}
S.~Conti, G.~Dolzmann, and C.~Kreisbeck.
\newblock Asymptotic behavior of crystal plasticity with one slip system in the
  limit of rigid elasticity.
\newblock {\em SIAM J. Math. Anal.}, 43(5):2337--2353, 2011.

\bibitem{CDK13a}
S.~Conti, G.~Dolzmann, and C.~Kreisbeck.
\newblock Relaxation and microstructure in a model for finite crystal
  plasticity with one slip system in three dimensions.
\newblock {\em Discrete Contin. Dyn. Syst. Ser. S}, 6(1):1--16, 2013.

\bibitem{CoT05}
S.~Conti and F.~Theil.
\newblock Single-slip elastoplastic microstructures.
\newblock {\em Arch. Ration. Mech. Anal.}, 178(1):125--148, 2005.

\bibitem{DaF15}
E.~Davoli and G.~A. Francfort.
\newblock A critical revisiting of finite elasto-plasticity.
\newblock {\em SIAM J. Math. Anal.}, 47(1):526--565, 2015.

\bibitem{Del18}
G.~Del~Piero.
\newblock On the decomposition of the deformation gradient in plasticity.
\newblock {\em J. Elasticity}, 131(1):111--124, 2018.

\bibitem{EnK21a}
D.~Engl and C.~Kreisbeck.
\newblock Asymptotic variational analysis of incompressible elastic strings.
\newblock {\em Proc. Roy. Soc. Edinburgh Sect. A}, 151(5):1487--1514, 2021.

\bibitem{FoHruMi03a}
M.~{Foss}, W.~J. {Hrusa}, and V.~J. {Mizel}.
\newblock {The Lavrentiev gap phenomenon in nonlinear elasticity.}
\newblock {\em {Arch. Ration. Mech. Anal.}}, 167(4):337--365, 2003.

\bibitem{GuFrAn10MTC}
M.~Gurtin, E.~E.~Fried, and L.~Anand.
\newblock {\em The Mechanics and Thermodynamics of Continua}.
\newblock Cambridge Univ. Press, New York, 2010.

\bibitem{Kro60}
E.~Kr\"{o}ner.
\newblock Allgemeine {K}ontinuumstheorie der {V}ersetzungen und
  {E}igenspannungen.
\newblock {\em Arch. Rational Mech. Anal.}, 4:273--334 (1960), 1960.

\bibitem{LeeLiu67FSEP}
E.~Lee and D.~Liu.
\newblock Finite‐strain elastic—plastic theory with application to
  plane‐wave analysis.
\newblock {\em J. Appl. Phys.}, 38:19--27, 1967.

\bibitem{Lee69}
E.~H. Lee.
\newblock Elastic-plastic deformation at finite strains.
\newblock {\em Journal of Applied Mechanics}, 36(1):1--6, 1969.

\bibitem{MieRou15RIST}
A.~Mielke and T.~Roub{\'\i}{\v{c}}ek.
\newblock {\em Rate-Independent Systems -- Theory and Application}.
\newblock Springer, New York, 2015.

\bibitem{Sca06}
L.~Scardia.
\newblock The nonlinear bending-torsion theory for curved rods as
  {$\Gamma$}-limit of three-dimensional elasticity.
\newblock {\em Asymptot. Anal.}, 47(3-4):317--343, 2006.

\bibitem{SimHug98CI}
J.~Simo and J.~Hughes.
\newblock {\em Computational Inelasticity}.
\newblock Springer, Berlin, 1998.

\end{thebibliography}
\end{document}